\documentclass[twoside,leqno,12pt]{amsart}
\usepackage{calc}

\usepackage{color}

\definecolor{darkgreen}{rgb}{0,0.5,0}

 \input xy
 \xyoption{all}
 \def\boldit#1{\textit{\textbf{#1}}}

\input prepictex   \input pictex    
\input postpictex
\def \ssize{\scriptscriptstyle}      

\def\smallsq#1{\plot 0 0  0.#1 0  0.#1 0.#1  0 0.#1  0 0 /}
\def \ssq{$\smallsq2$}
\def \bul{$\ssize\bullet$}
\def\num#1{$\scriptscriptstyle\sf#1$}

\def\qed{\hfill$\square$}
\numberwithin{equation}{section}
\newtheorem{thm}[equation]{\sc Theorem}
\newtheorem{lem}[equation]{\sc Lemma}
\newtheorem{cor}[equation]{\sc Corollary}
\newtheorem{prop}[equation]{\sc Proposition}

\newtheoremstyle{notation}{3pt}{3pt}{}{}{\itshape}{:}{.5em}{\thmname{#1}}
\theoremstyle{notation}

\newtheorem{rem}{\it Remark}

\newtheorem{defin}{\it Definition}
\newtheorem{ex}{\it Example}

\def\type{\mbox{\rm type\,}}
\def\Cok{\mbox{\rm Cok}}

\def\Hom{\mbox{\rm Hom}}

\def\rad{\mbox{\rm rad}}
\def\soc{\mbox{\rm soc\,}}

\def\Ker{\mbox{\rm Ker\,}}

\def\incl{\mbox{\rm incl}}

\def\len{\mbox{\rm len\,}}
\def\Sub{\mbox{\rm Sub}}

\newcounter{boxsize}
\newcounter{boxlength}
\newcounter{tempcounter}
\newcommand{\smallentryformat}{\scriptstyle\sf}

\newcommand\smbox{\put(0,0){\line(1,0){\value{boxsize}}}%
  \put(\value{boxsize},0){\line(0,1){\value{boxsize}}}%
  \put(0,0){\line(0,1){\value{boxsize}}}%
  \put(0,\value{boxsize}){\line(1,0){\value{boxsize}}}}
\newcommand\numbox[1]{\put(0,0)\smbox%
  \put(0,0){\makebox(\value{boxsize},\value{boxsize})[c]{%
      $\smallentryformat#1$}}}
\newcommand\singlebox[1]{\raisebox{-.4ex}{\begin{picture}(4,0)\setcounter{boxsize}{3}%
    \put(0,0)\smbox%
    \put(0,0){\makebox(\value{boxsize},\value{boxsize})[c]{%
      $\scriptstyle\sf#1$}}\end{picture}}}

\newcommand\boxedentry[3]{
  \raisebox{-.4ex}{%
    \setcounter{boxlength}{#1}%
    \begin{picture}(\value{boxlength},3)%
      \put(0,0){\line(1,0){\value{boxlength}}}%
      \put(\value{boxlength},0){\line(0,1){\value{boxsize}}}%
      \put(0,0){\line(0,1){\value{boxsize}}}%
      \put(0,\value{boxsize}){\line(1,0){\value{boxlength}}}%
      \put(0,0){\makebox(\value{boxlength},\value{boxsize})[c]{$\scriptstyle#2$}}
      \put(\value{boxlength},1.5){\line(1,0){1}}%
      \put(\value{boxlength},0){\makebox(0,\value{boxsize})[l]{$\;\scriptstyle#3$}}%
    \end{picture}
  }%
  \hspace{1ex}
}

\def\sbullet{\makebox(0,0){$\scriptstyle\bullet$}}
\def\arr#1#2{\arrow <2mm> [0.25,0.75] from #1 to #2}

\def\ssize{\scriptscriptstyle}

\newcommand\boxes[2]{\ifthenelse{#2=3}{$\scriptstyle P_2^{#1}$}{%
                                       $\scriptstyle P_{#2}^{#1}$}}

\newcommand\Mone{\begin{picture}(11,15)
    \multiput(0,0)(0,3)5{\smbox}
    \multiput(4,3)(0,3)3{\smbox}
    \multiput(8,3)(0,3)2{\smbox}
    \put(1.5,10.5)\sbullet
    \put(9.5,7.5)\sbullet
  \end{picture}
}
\newcommand\MoneD{\begin{picture}(11,15)
    \multiput(0,0)(0,3)5{\smbox}
    \multiput(4,3)(0,3)3{\smbox}
    \multiput(8,3)(0,3)2{\smbox}
    \put(1.5,1.5)\sbullet
    \put(5.5,10.5)\sbullet
  \end{picture}
}
\newcommand\LRone{\begin{picture}(9,15)
    \multiput(0,0)(0,3)5{\smbox}
    \multiput(3,6)(0,3)3{\smbox}
    \multiput(6,9)(0,3)2{\smbox}
    \put(6,12){\numbox1}
    \put(3,9){\numbox1}
    \put(6,9){\numbox2}
    \put(3,6){\numbox2}
    \put(0,3){\numbox3}
    \put(0,0){\numbox4}
  \end{picture}
}
\newcommand\STone{\begin{picture}(9,15)
    \multiput(0,0)(0,3)5{\smbox}
    \multiput(3,6)(0,3)3{\smbox}
    \multiput(6,9)(0,3)2{\smbox}
    \put(6,12){\numbox2}
    \put(3,9){\numbox4}
    \put(6,9){\numbox1}
    \put(3,6){\numbox3}
    \put(0,3){\numbox2}
    \put(0,0){\numbox1}
  \end{picture}
}
\newcommand\STtwo{\begin{picture}(9,15)
    \multiput(0,0)(0,3)5{\smbox}
    \multiput(3,6)(0,3)3{\smbox}
    \multiput(6,9)(0,3)2{\smbox}
    \put(6,12){\numbox4}
    \put(3,9){\numbox3}
    \put(6,9){\numbox2}
    \put(3,6){\numbox1}
    \put(0,3){\numbox2}
    \put(0,0){\numbox1}
  \end{picture}
}
\newcommand\LRthree{\begin{picture}(9,15)
    \multiput(0,0)(0,3)5{\smbox}
    \multiput(3,6)(0,3)3{\smbox}
    \multiput(6,9)(0,3)2{\smbox}
    \put(6,12){\numbox1}
    \put(3,9){\numbox1}
    \put(6,9){\numbox2}
    \put(3,6){\numbox3}
    \put(0,3){\numbox2}
    \put(0,0){\numbox4}
  \end{picture}
}
\newcommand\Mtwo{\begin{picture}(10,15)
    \multiput(0,0)(0,3)5{\smbox}
    \multiput(3,6)(0,3)2{\smbox}
    \multiput(7,3)(0,3)3{\smbox}
    \multiput(1.5,10.5)(3,0)2{\sbullet}
    \put(1.5,10.5){\line(1,0)3}
    \put(8.5,7.5){\sbullet}
  \end{picture}
}
\newcommand\MtwoD{\begin{picture}(10,15)
    \multiput(0,0)(0,3)5{\smbox}
    \multiput(3,3)(0,3)2{\smbox}
    \multiput(7,3)(0,3)3{\smbox}
    \multiput(1.5,7.5)(3,0)2{\sbullet}
    \put(1.5,7.5){\line(1,0)3}
    \put(8.5,4.5){\sbullet}
  \end{picture}
}
\newcommand\Mthree{\begin{picture}(9,15)
    \multiput(0,0)(0,3)5{\smbox}
    \multiput(3,3)(0,3)3{\smbox}
    \multiput(6,6)(0,3)2{\smbox}
    \multiput(1.5,10.5)(3,0)2{\sbullet}
    \put(1.5,10.5){\line(1,0)3}
    \multiput(4.5,7.5)(3,0)2{\sbullet}
    \put(4.5,7.5){\line(1,0)3}
  \end{picture}
}
\newcommand\MthreeD{\begin{picture}(9,15)
    \multiput(0,0)(0,3)5{\smbox}
    \multiput(3,3)(0,3)3{\smbox}
    \multiput(6,3)(0,3)2{\smbox}
    \multiput(1.5,7.5)(3,0)3{\sbullet}
    \put(1.5,7.5){\line(1,0)6}
    \put(7.5,4.5){\sbullet}
  \end{picture}
}
\newcommand\LRoneD{\begin{picture}(9,15)
    \multiput(0,0)(0,3)5{\smbox}
    \multiput(3,6)(0,3)3{\smbox}
    \multiput(6,9)(0,3)2{\smbox}
    \put(6,12){\numbox1}
    \put(6,9){\numbox2}
    \put(3,6){\numbox3}
    \put(0,0){\numbox1}
  \end{picture}
}
\newcommand\SToneD{\begin{picture}(9,15)
    \multiput(0,0)(0,3)5{\smbox}
    \multiput(3,6)(0,3)3{\smbox}
    \multiput(6,9)(0,3)2{\smbox}
    \put(6,12){\numbox3}
    \put(6,9){\numbox2}
    \put(3,6){\numbox1}
    \put(0,0){\numbox1}
  \end{picture}
}
\newcommand\LRtwoD{\begin{picture}(9,15)
    \multiput(0,0)(0,3)5{\smbox}
    \multiput(3,6)(0,3)3{\smbox}
    \multiput(6,9)(0,3)2{\smbox}
    \put(6,12){\numbox1}
    \put(6,9){\numbox2}
    \put(3,6){\numbox1}
    \put(0,0){\numbox3}
  \end{picture}
}
\newcommand\STthreeD{\begin{picture}(9,15)
    \multiput(0,0)(0,3)5{\smbox}
    \multiput(3,6)(0,3)3{\smbox}
    \multiput(6,9)(0,3)2{\smbox}
    \put(6,12){\numbox3}
    \put(6,9){\numbox1}
    \put(3,6){\numbox2}
    \put(0,0){\numbox1}
  \end{picture}
}

\begin{document}
\thispagestyle{empty}
\color{black}
\phantom m\vspace{-2cm}

\bigskip\bigskip
\begin{center}
  {\large\bf The socle tableau as a dual version\\[.5ex] of the Littlewood-Richardson tableau}

  \bigskip \today
\end{center}

\smallskip

\begin{center}
Justyna Kosakowska and Markus Schmidmeier

\medskip

  {\it Dedicated to the memory of Professor Andrzej Skowro\'nski}


\bigskip \parbox{10cm}{\footnotesize{\bf Abstract:}
    Like the LR-tableau, a socle tableau is given as a skew diagram with certain
    entries.  Unlike in the LR-tableau, 
    the entries in the socle tableau are weakly increasing in each row,
    strictly increasing in each column and satisfy a modified lattice
    permutation property.
    In the study of embeddings of a subgroup in a finite abelian $p$-group,
    socle tableaux occur as isomorphism invariants, they are given
    by the socle series of the subgroup.  We show that each
    socle tableau can be realized by some embedding.
    Moreover, the socle tableau
    of an embedding and the LR-tableau of the
    dual embedding determine each other
     --- a correspondence which appears to be related to
      the tableau switching studied by Benkart, Sottile and Stroomer.
      We illustrate how the entries in the socle tableau position
      the object within the Auslander-Reiten quiver.  Like the LR-tableau,
      the socle tableau determines the irreducible component of the object
      in representation space.}

\medskip \parbox{10cm}{\footnotesize{\bf MSC 2010:}

    05E10, 
    20K27, 
  47A15,  
  16G20 
}

\medskip \parbox{10cm}{\footnotesize{\bf Key words:} 
subgroup embedding; invariant subspace; Littlewood-Richardson tableau }
\end{center}

\section{Introduction}

  \subsection{Combinatorics.}

  The Littlewood-Richardson (LR) coefficient occurs in a meaningful way in
  a variety of areas in algebra: for example as the structure constant
  of the product of Schur polynomials in the ring of symmetric functions;
  as the multiplicity of a given irreducible representation of the
  symmetric group in a decomposition of the tensor product of two others;
  or as the number of irreducible components of invariant subspace
  varieties \cite{fulton,ks-box,ks-poles,leeu,macd}.

  \medskip
  Combinatorially, the LR-coefficient $c_{\alpha,\gamma}^\beta$ can be
  computed as the number of LR-tableaux of shape $(\alpha,\beta,\gamma)$,
  where $\alpha,\beta,\gamma$ are given partitions.
  In this paper we introduce a new kind of tableau,
  which we call the \boldit{socle tableau}.
  Here, the entries in the tableau are weakly decreasing in each row,
  strictly decreasing in each column, and satisfy a third condition
  which corresponds to the lattice permutation property.

  \medskip
  The number of socle tableaux of shape $(\alpha,\beta,\gamma)$ is
  is equal to the number of LR-tableaux of the same shape.
  However to the knowledge of the authors,
  there is no ``natural'' one-to-one correspondence between
  LR-tableaux and socle tableaux of the same shape.

  \label{example-two}
  \medskip
  It is well-known that the number $c_{\alpha,\gamma}^\beta$ of
  LR-tableaux of shape $(\alpha,\beta,\gamma)$ equals the number
  $c_{\gamma,\alpha}^\beta$  of LR-tableaux of shape $(\gamma,\beta,\alpha)$.
  In Section~\ref{sec-lr-soc} we present an explicit combinatorial
  one-to-one correspondence between the set of socle tableaux
  of shape $(\alpha,\beta,\gamma)$ and the set of LR-tableaux of shape
  $(\gamma,\beta,\alpha)$.  For example,
  $$\Sigma_2:\;\raisebox{-5mm}\STtwo\quad\text{corresponds to}\quad
  \Gamma_2^*:\;\raisebox{-5mm}\LRtwoD.$$

 As a~consequence we can interpret the Littlewood-Richardson coefficient
 $c_{\alpha,\gamma}^\beta$ as the number
 of socle tableaux of shape $(\alpha,\beta,\gamma)$
 (Corollary \ref{cor-LR-coeff}).

  \subsection{Algebra.}
  Both LR- and socle
  tableaux occur as isomorphism invariants of embeddings of the form
  $(A\subset B)$ where $B$ is a finite length module over a
  discrete valuation ring $\Lambda$ and $A$ is a submodule of $B$.
  The shape parameters $\alpha,\beta,\gamma$
  are the partitions which describe the isomorphism types of the
  $\Lambda$-modules $A$, $B$, $B/A$.
  While the LR-tableau encodes the
  isomorphism types of the quotients $B/\rad^\ell A$, for integers $\ell\geq0$,
  given by the radical series of $A$,
  the socle tableau encodes the isomorphism types of the quotients
  $B/\soc^\ell A$ given by the socle series of $A$.  Hence the name.

  \medskip
  We present examples to illustrate that
  the LR-tableau of an embedding does not determine the socle
  tableau of that embedding, and conversely.

  \medskip
  Certain embeddings of $\Lambda$-modules, the \boldit{pickets},
  play a key role.
  For natural numbers $1\leq m$, $0\leq \ell\leq m$,
  denote by $P_\ell^m$ the embedding $(A\subset B)$ where
  $B =\Lambda/\rad^m\Lambda$ is the
  cyclic $\Lambda$-module of (composition)
  length $m$ and $A=\soc^\ell B$ the unique submodule of $B$ of length $\ell$.
  It turns out that either the socle tableau or the dual LR-tableau
  of an embedding $X$ determines and is determined by the Hom-matrix
  $H=(h_\ell^m)$, where $h_\ell^m$ measures the length of the homomorphism
  space $\Hom(P_\ell^m,X)$
  (Theorem~\ref{thm-lr-soc-hom}).
  In Subsection~\ref{sec-lr-soc} we will give an explicit
  combinatorial construction how to obtain the socle
  tableau from the corresponding dual LR-tableau, and conversely.

  \subsection{Contents of the sections.}

  \smallskip
  In Section~\ref{sec-notation} we define the socle tableau and discuss
  the modified lattice permutation property.
  We introduce embeddings of modules over a discrete valuation domain
  and the associated LR- and socle tableaux.
  Examples show that the LR-tableau of an embedding does not determine
  the socle tableau, and conversely.

  \smallskip
  In Section~\ref{sec-soc-tab} we verify that the socle tableau of
  an embedding does indeed satisfy the properties of a socle tableau.
  Our main result is the version of the Green-Klein Theorem
  (\cite{klein1,macd}) for
  socle tableaux:  For a~given socle tableau $\Sigma$ we explicitely
  construct an embedding with this socle tableau (Theorem~\ref{thm-kleinST}).

  \smallskip
  The Hom-matrix $(h_\ell^m)_{\ell,m}$ of an embedding $X$ encodes the
  lengths of the homomorphism spaces $\Hom(P_\ell^m,X)$ between pickets
  and the given embedding. 
  In Section~\ref{sec-soc-lr-hom} we show that
  socle tableau, hom-matrix and the LR-tableau of the dual embedding
  determine each other.

  \smallskip
  This correspondence between socle tableau and dual LR-tableau
  appears to be related to the tableau switching studied by
  Benkart, Sotille and Stroomer.  We give an example and state the
  conjecture.  In the remaining parts of Section~\ref{sec-applications}
  we describe how each entry in a certain row in the socle tableau
  specifies the position of the object in the Auslander-Reiten quiver;
  we also illustrate how the socle tableau, like the LR-tableau,
  determines the irreducible component of an object within the representation
  space.

  \section{Notation and Examples}
  \label{sec-notation}
  
\subsection{The LR-tableau and the socle tableau}

Let $\alpha$, $\beta$, $\gamma$ be partitions.  For LR- and socle
tableaux of shape $(\alpha,\beta,\gamma)$ to exist, there are two
necessary conditions.  First, 
$\gamma\leq\beta$ in the sense that the Young diagram for $\gamma$
(which has columns of length $\gamma_1,\gamma_2,\ldots$)
fits into the Young diagram for $\beta$ 
(that is, $\gamma_i\leq \beta_i$ holds for all $i$),
so that we can consider the
skew diagram $\beta\setminus\gamma$.  The second condition is
that $|\alpha|+|\gamma|=|\beta|$ so that we can fill the skew diagram
$\beta\setminus\gamma$ with $\alpha'_1$ entries 1, $\alpha'_2$ entries 2,
etc.  Here, $\alpha'$ is the transpose of the partition $\alpha$, so
$\alpha'_i$ is the length of the $i$-th row in the Young diagram for $\alpha$.

\begin{defin}\sloppypar
  A \boldit{Littlewood-Richardson (LR-) tableau}
  of shape $(\alpha,\beta,\gamma)$
  is a filling of the skew diagram $\beta\setminus\gamma$ with
  $\alpha'_1$ boxes \singlebox1, $\alpha_2'$ boxes \singlebox2 etc.\ such that
  \begin{itemize}
  \item[(LR-1)] in each row, the entries are weakly increasing,
  \item[(LR-2)] in each column, the entries are strictly increasing,
  \item[(LR-3)] (lattice permutation property) for each vertical line,
    and for each natural number $\ell$, there are at most as many
    entries $\ell+1$ on the right hand side of the line as there are
    entries $\ell$.
  \end{itemize}
\end{defin}

Socle tableaux are defined correspondingly:

\begin{defin}
  A \boldit{socle tableau} of shape $(\alpha,\beta,\gamma)$ is a filling
  of the skew diagram $\beta\setminus\gamma$ with $\alpha'_1$ boxes
  \singlebox1, $\alpha'_2$ boxes \singlebox2, etc.\ such that 
  \begin{itemize}
  \item[(ST-1)] in each row, the entries are weakly decreasing,
  \item[(ST-2)] in each column, the entries are strictly decreasing,
  \item[(ST-3)] for each vertical line, and each natural number $\ell$,
    there are at most as many entries $\ell+1$ on the left hand side
    of the line as there are entries $\ell$.
  \end{itemize}
\end{defin}

\begin{ex}
  There are two LR-tableaux and two socle tableaux of shape
  $(42, 532, 31)$ (for the notation, see below):
$$ \Gamma_1:\;\raisebox{-5mm}{\LRone},\qquad
\Gamma_3:\;\raisebox{-5mm}{\LRthree},\qquad
\Sigma_1:\;\raisebox{-5mm}{\STone}, \qquad
\Sigma_3:\;\raisebox{-5mm}{\STtwo}$$
\end{ex}

\smallskip
  In all diagrams we number rows from up down and columns from left to right.

\medskip
Let $\alpha$, $\beta$, $\gamma$ be partitions and put $\alpha_1=s$.
Formally, an LR-tableau $\Gamma$
can be given by an increasing
sequence $(\gamma^{(0)},\gamma^{(1)},\ldots,\gamma^{(s)})$ of partitions
where the Young diagram for $\gamma^{(i)}$ consists of all empty boxes and
boxes with entries at most $i$.  In particular, $\gamma^{(0)}=\gamma$ and
$\gamma^{(s)}=\beta$.

\medskip\sloppypar
Similarly, a socle tableau $\Sigma$ is given by a decreasing sequence
$(\sigma^{(0)},\sigma^{(1)},\ldots,\sigma^{(s)})$ of partitions where the
Young diagram for $\sigma^{(i)}$ consists of all empty boxes and boxes
with entries strictly greater than $i$.  In particular, $\sigma^{(0)}=\beta$ and
$\sigma^{(s)}=\gamma$.

\begin{lem}
  \label{lem-lpp-equiv}
  In the definition of the socle tableau, assuming conditions
  {\rm (ST-1)} and {\rm (ST-2)}, the property {\rm (ST-3)}
  is equivalent to each of the following
  conditions:
  \begin{itemize}
  \item[(ST-3$'$)] {\rm for each row, and each natural number $\ell$,
    there are at most as many entries $\ell+1$ in this row and underneath
    as there are entries $\ell$ strictly underneath the row.}
  \item[(ST-3$''$)] {\rm for each natural number $\ell$, 
    there exists a one-to-one map $\varphi_\ell$ from the set of boxes with entry $\ell+1$
    to the set of boxes with entry $\ell$ such that $\varphi_\ell(b)$
    is in the same column as $b$ provided this column contains a box with entry $\ell$,
    and in a column to the left of the column of $b$ otherwise.}
  \end{itemize}
\end{lem}

\begin{proof}
  It is clear that conditions (ST-3) and (ST-3$''$) are equivalent.
  Let $\Sigma$ be a~socle tableau. For a vertical line $L$ and a row $R$ write
  \begin{eqnarray*}
    L_\ell & = & \#\{\singlebox \ell \;\text{in $\Sigma$ on the left of $L$}\},\\
    R_\ell & = & \#\{\singlebox \ell \;\text{in $\Sigma$ strictly underneath $R$}\},\\
    R_\ell' & = & \#\{\singlebox \ell\;\text{in $\Sigma$ in row $R$ or underneath}\}.
  \end{eqnarray*}
      {\rm (ST-3)$\implies$(ST-3$'$):} Given a row $R$ and a natural number
      $\ell$, let $L$ be the vertical line through $\Sigma$ which
      separates the entries greater than $\ell$ in row $R$
        from those less than or equal to $\ell$.
      Then
      $R'_{\ell+1}=L_{\ell+1}\leq L_\ell=R_\ell$.

      \medskip
          {\rm (ST-3$'$)$\implies$(ST-3):} Given a vertical line $L$
          and a natural number $\ell$, let $R$ be the first row which contains
          an entry $\ell+1$ on the left of $L$.  (If there is no such
          entry then there is nothing to show.)  Let $u$ be the number
          of entries $\ell+1$ in row $R$ on the right of $L$.  Then
          $L_{\ell+1}=R'_{\ell+1}-u\leq R_\ell-u\leq L_\ell$ where the last
          inequality holds since there are at most $u$ entries $\ell$
          on the right of $L$ and underneath row $R$.

\end{proof}

  \subsection{$\Lambda$-modules and embeddings}

Let $\Lambda$ be a~(commutative) discrete valuation ring
with maximal ideal generator $p$
and radical factor field $k=\Lambda/(p)$.
In this paper, we assume that all $\Lambda$-modules
have finite length.
Examples of discrete valuation rings
include the localization $\Lambda=\mathbb Z_{(p)}$ 
of the ring of integers at the prime ideal $(p)$,
then $\Lambda$-modules are the finite
abelian $p$-groups; and the power series ring $\Lambda=k[[T]]$ 
with coefficients in a field $k$, then $\Lambda$-modules are the finite
dimensional modules over the polynomial ring $k[T]$ which are annihilated
by some power of $T$.

\medskip
It is well known
there is a one-to-one correspondence between the set of
isomorphism classes of $\Lambda$-modules and the set of partitions:
For a partition $\alpha=(\alpha_1,\ldots,\alpha_n)$ where
$\alpha_1\geq\cdots\geq\alpha_n\geq1$ are natural numbers,
we denote by $N_\alpha$ the $\Lambda$-module 
$$N_\alpha = \Lambda/(p^{\alpha_1})\oplus \cdots \oplus \Lambda/(p^{\alpha_n})$$
and write $\alpha=\type(A)$ if $A\cong N_\alpha$.
  We denote by $\len A$ the (composition)
  length of the $\Lambda$-module $A$; in particular, the length
  of $N_\alpha$ as a $\Lambda$-module equals
  the length $|\alpha|=\alpha_1+\cdots+\alpha_n$
  of $\alpha$ as a partition.

\medskip
An \boldit{embedding} 
$(A\subset B)$ consists of a $\Lambda$-module $B$ and a submodule $A$ of $B$.
By  $\mathcal{S}=\mathcal S(\Lambda)$ we denote the category of all embeddings,
with homomorphisms given by commutative diagrams.

\medskip
  For a $\Lambda$-module $B$, the multiplication by
  the radical generator $p\in\Lambda$
  gives rise to two maps $\Sub_B
    \to \Sub_B$
      where $\Sub_B$
        is the set of $\Lambda$-submodules
  of $B$.

  \begin{eqnarray*}
    p_B: & \Sub_B \to \Sub_B, & A\mapsto \{pa:a\in A\}\\
    p_B^{-1}: & \Sub_B \to \Sub_B, & A\mapsto \{b\in B:pb\in A\}
  \end{eqnarray*}
  In particular, if $A\subset B$ is an embedding, then the layers
  of the radical series and the socle series of $A$ are the submodules
  of $B$ given by
  $$\rad^mA = p_B^m(A),\quad \soc^\ell A=p_A^{-\ell}(0).$$

  \medskip
  For two $\Lambda$-modules $B,C$,
  and two embeddings $X,Y\in\mathcal S(\Lambda)$,
  the homomorphism groups are $\Lambda$-modules and we write
  $$\hom_\Lambda(B,C)=\len\Hom_\Lambda(B,C),\quad
  \hom_{\mathcal S}(X,Y)=\len\Hom_{\mathcal S}(X,Y).$$

\subsection{Tableaux given by an embedding}
\label{sub-embedding}
Let $(A\subset B)$ be an embedding, and denote by 
$\alpha$, $\beta$, and $\gamma$ the partition type of the $\Lambda$-module
$A$, $B$, and $B/A$, respectively.  Let $s=\alpha_1$, so $p^sA=0$.

\medskip
The radical sequence for $A$,
$$0=\rad^sA\subset \rad^{s-1}A\subset\cdots\subset\rad A\subset A,$$
gives rise to a sequence of epimorphisms,
$$B=B/\rad^sA\to B/\rad^{s-1}A\to\cdots\to B/\rad A\to B/A,$$
and hence to an increasing sequence of partitions,
$$\beta=\gamma^{(s)}\geq\gamma^{(s-1)}\geq\cdots
\geq\gamma^{(1)}\geq\gamma^{(0)} =\gamma,$$
where $\gamma^{(i)}$ is the partition type of $B/\rad^iA$.
The corresponding tableau $\Gamma=(\gamma^{(i)})$ is an LR-tableau,
according to the Theorem by Green and Klein, see Section \ref{sec-soc-tab}. We call $\Gamma$ \boldit{the LR-tableau of the embedding} $(A\subseteq B)$.

\medskip
Dually, the socle sequence for $A$,
$$0\subset \soc A\subset\cdots\subset \soc^{s-1}A\subset\soc^sA=A,$$
gives rise to a sequence of epimorphisms,
$$B\to B/\soc A\to \cdots\to B/\soc^{s-1}A\to B/\soc^sA=B/A,$$
and hence to a decreasing sequence of partitions,
$$\beta =\sigma^{(0)} \geq\sigma^{(1)}\geq\cdots
\geq\sigma^{(s-1)}\geq\sigma^{(s)}=\gamma,$$
where $\sigma^{(i)}$ is the partition type of $B/\soc^iA$.
We will see in the next section that $\Sigma=(\sigma^{(i)})$
is a socle tableau.
We call $\Sigma$ \boldit{the socle tableau of the embedding} $(A\subseteq B)$.

\bigskip
\subsection{Examples.}

Let $m$ be a natural number and $0\leq \ell\leq m$.  The picket
$P_\ell^m$ is given by the embedding $(\soc^\ell P^m\subset P^m)$
or $(\rad^{m-\ell}P^m\subset P^m)$ where $P^m=\Lambda/(p^m)$.
We picture the picket as a column of $m$ boxes (representing the
$\Lambda$-module $P^m$) and put a dot
in the $(m-\ell)$-th box from the top to represent the generator
of the submodule:  The top box stands for the element $1+(p^m)$ in $P^m$,
the second box for $p+(p^m)$ etc.
The partition type for $P^m_\ell$ is easily computed as
$\beta=(\len P_m)
= (m)$, $\alpha=\len \soc^\ell P^m=(\ell)$,
$\gamma=\len P_m/\soc^\ell P^m=(m-\ell)$, or $\gamma=()$ if $\ell=m$.

\medskip
The modules in the radical series of the submodule have length
$\len \rad^i(\soc^\ell P^m)=\ell-i$ for $0\leq i\leq \ell$
hence the partitions defining the LR-tableau are $\gamma^{(i)}=
\type P^m/\rad^i(\soc^\ell P^m)=(m-\ell+i)$.
Similarly, the modules in the socle series of the submodule have
length $\soc^i(\soc^\ell P^m)=i$ for $0\leq i\leq \ell$, so the partitions
in the socle tableau are $\sigma^{(i)}=\type P^m/\soc^i(\soc^\ell P^m)=(m-i)$.

\medskip
Here we picture the embedding $P_4^5$ together with its LR-tableau
$\Gamma_4^5$ and its socle tableau $\Sigma_4^5$.
$$P_4^5: \raisebox{-5mm}{\begin{picture}(3,15)
    \multiput(0,0)(0,3)5{\smbox}
    \put(1.5,10.5)\sbullet
\end{picture}}\;,
\qquad
\Gamma_4^5: \raisebox{-5mm}{\begin{picture}(3,15)
    \multiput(0,0)(0,3)5{\smbox}
    \put(0,9){\numbox1}
    \put(0,6){\numbox2}
    \put(0,3){\numbox3}
    \put(0,0){\numbox4}
\end{picture}}\;,
\qquad\Sigma_4^5: \raisebox{-5mm}{\begin{picture}(3,15)
    \multiput(0,0)(0,3)5{\smbox}
    \put(0,9){\numbox4}
    \put(0,6){\numbox3}
    \put(0,3){\numbox2}
    \put(0,0){\numbox1}
\end{picture}}
$$

Here are some more examples.
$M_1$ is the direct sum $P_4^5\oplus P_0^3\oplus P_2^2$.  
In $M_2$ and $M_3$, the ambient space is $P^5\oplus P^3\oplus P^2$,
say generated by $b$, $b'$, and $b''$; in $M_2$, the subspace
generators are $pb+b''$ and $pb'$; in $M_3$, the submodule is generated
by $pb+b'$ and $pb'+pb''$.

\medskip
Each of the embeddings $M_i$ has partition type
$\alpha=(42)$, $\beta=(532)$, $\gamma=(31)$.
$$M_1:\raisebox{-5mm}{\Mone},\qquad
M_2:\raisebox{-5mm}{\Mtwo}, \qquad
M_3:\raisebox{-5mm}{\Mthree}.$$
The LR- and socle tableaux are as follows.
$$\Sigma_1:\raisebox{-5mm}{\STone}, \qquad
\Gamma_1=\Gamma_2:\raisebox{-5mm}{\LRone},\qquad
\Sigma_2=\Sigma_3:\raisebox{-5mm}{\STtwo},\qquad
\Gamma_3:\raisebox{-5mm}{\LRthree}$$

We observe that embeddings with the same LR-tableau may have different
socle tableaux, and conversely.

\subsection{Duality}

\begin{defin}
  Let $E$ be the injective envelope of the simple $\Lambda$-module
  $\Lambda/(p)$.  We write $DB=\Hom_\Lambda(B,E)$ for the dual of the
  $\Lambda$-module $B$.  
  For an embedding $X=(A\subset B)\in \mathcal S(\Lambda)$, the
  \boldit{dual embedding} $X^*$ is given by
  the inclusion $\Hom_\Lambda(\pi,E): DC\to DB$
  where $C=B/A$ and $\pi:B\to C$ is the canonical map.

  Note that if the embedding $(A\subset B)$ has partition type
  $(\alpha,\beta,\gamma)$, then the dual embedding has type
  $(\gamma,\beta,\alpha)$.  We define
  the \boldit{dual LR-tableau} of the embedding $(A\subset B)$ as the
  LR-tableau of the dual of the embedding, $\Gamma^*_X= \Gamma_{X^*}$,
  it is a tableau of shape $(\gamma,\beta,\alpha)$.
\end{defin}

\begin{ex} We present the dual embeddings for $M_1$, $M_2$, $M_3$
from the previous subsection.  
$$M_1^*:\raisebox{-5mm}{\MoneD},\qquad
M_2^*:\raisebox{-5mm}{\MtwoD}, \qquad
M_3^*:\raisebox{-5mm}{\MthreeD}$$
They have the following LR- and socle tableaux.
$$\Gamma_1^*:\raisebox{-5mm}{\LRoneD}, \qquad
\Sigma_1^*=\Sigma_2^*:\raisebox{-5mm}{\SToneD},\qquad
\Gamma_2^*=\Gamma_3^*:\raisebox{-5mm}{\LRtwoD},\qquad
\Sigma_3^*:\raisebox{-5mm}{\STthreeD}$$

We have seen in the previous subsection
that the socle tableau of an embedding does not determine
its LR-tableau or conversely.
However, the socle tableau of an embedding
and the dual LR-tableau do determine each other.
In Section~\ref{sec-lr-soc} we describe how to obtain one tableau from
the other combinatorially.
\end{ex}

\section{The Green-Klein Theorem revisited}
\label{sec-soc-tab}

This section is inspired and motivated by the Theorem of Green and Klein,
which we present in the version for $\Lambda$-modules.

\begin{thm}[\protect{\cite{klein1,macd}}]
  Given partitions $\alpha$, $\beta$, $\gamma$,
  there exists a short exact sequence of $\Lambda$-modules
  $$
    0\longrightarrow A \longrightarrow B\longrightarrow C \longrightarrow 0
  $$
  where $A$, $B$, $C$ have partition type $\alpha$, $\beta$, $\gamma$,
  respectively, if and only if there exists an LR-tableau of
  shape $(\alpha,\beta,\gamma)$.\qed
\end{thm}

  Our aim is to show the corresponding result for socle tableaux.

\begin{thm}
  Given partitions $\alpha$, $\beta$, $\gamma$,
  there exists a short exact sequence of $\Lambda$-modules
  \begin{equation}\label{eq-ses}
    0\longrightarrow A \longrightarrow B\longrightarrow C \longrightarrow 0
  \end{equation}
  where $A$, $B$, $C$ have partition type $\alpha$, $\beta$, $\gamma$,
  respectively, if and only if there exists a socle tableau of
  shape $(\alpha,\beta,\gamma)$.
\end{thm}

In the first subsection we show that the socle tableau of the embedding
$(A\subset B)$ 
as defined in Subsection~\ref{sub-embedding}
does satisfy conditions (ST-1) through (ST-3). 
In the second subsection
we show that any such tableau can be realized as the socle tableau
of an embedding.

  \subsection{Conditions (ST-1) to (ST-3) are satisfied.}

  Let $\Sigma$ be the socle tableau of an embedding $(A\subset B)$,
  so there are partitions $\alpha=\type A$, $\beta=\type B$, $\gamma=\type B/A$
  such that $\Sigma$ is the skew diagram of shape $\beta\setminus\gamma$.
    If the socle tableau $\Sigma$ is given by partition sequence
    $(\sigma^{(i)})$ where $\sigma^{(i)}=\type B/\soc^iA$,
    then for each natural number $\ell$, the entries
    \singlebox\ell\ occur in the skew diagram
    $\sigma^{(\ell-1)}\setminus\sigma^{(\ell)}$.
  
  \medskip
  We adapt to our situation the proof of the analogous fact for LR-tableaux,
  see \cite[II.3.4]{macd}.  Let $s=\alpha_1$.

  \medskip
  Obviously we have $\sigma^{(s)}=\gamma$ and 
  $\sigma^{(0)}=\beta$, because 
  $\soc^sA=A$ and $\soc^0A=0$. Moreover, there is a~short exact sequence
  $$
  0\to \soc^{\ell}A/\soc^{\ell-1}A \to B/\soc^{\ell-1} A \to B/\soc^{\ell}A \to 0
  $$
  and therefore $\sigma^{(\ell)}\subseteq\sigma^{(\ell-1)}$ for all $\ell$,
  see \cite[II.3.1]{macd}. This shows (ST-1) and the fact that
  there are $\alpha'_\ell=\len(\soc^\ell A/\soc^{\ell-1}A)$ boxes
  \singlebox\ell\ in $\Sigma$.

  \smallskip
  Since
  $\soc^{\ell} A/\soc^{\ell-1}A$ is a~semisimple module,
  $\sigma^{(\ell-1)}\setminus\sigma^{(\ell)}$ is a~horizontal strip
  see \cite[II.3.3]{macd}. Thus condition (ST-2) is also satisfied.

  \medskip
  It remains to show property (ST-3), or equivalently, (ST-3$'$)
  (Lemma~\ref{lem-lpp-equiv}).

  \begin{lem}
    For $\ell\geq2$, the map
    $$\mu_p:\quad \frac{\soc^\ell A}{\soc^{\ell-1}A}\longrightarrow
    \frac{\soc^{\ell-1}A}{\soc^{\ell-2}A}$$
    given by multiplication by $p$ is a monomorphism.
  \end{lem}

  \begin{proof}
    From the description of the socle as $\soc^\ell A=p_A^{-\ell}(0)$,
    we obtain for $a\in\soc^\ell A$
    that $pa\in \soc^{\ell-1}A$, hence the map $\mu_p$ is defined.
    Moreover, if $pa\in\soc^{\ell-2}A=p_A^{-(\ell-2)}(0)$, then
    $a\in p^{-(\ell-1)}_A(0)=\soc^{\ell-1}A$; so $\mu_p$ is a monomorphism.
  \end{proof}

  We will use the following easy generalization.

  \begin{cor}
    \label{cor-mono}
    For an embedding $A\subset B$ and natural numbers $\ell\geq2$, $s\geq1$,
    the map
    $$\mu_p:\quad\frac{\soc^\ell A\cap\rad^{s-1}B}{\soc^{\ell-1}A\cap\rad^{s-1}B}
    \longrightarrow
    \frac{\soc^{\ell-1}A\cap\rad^sB}{\soc^{\ell-2}A\cap\rad^sB}$$
    given by multiplication by $p$ is a monomorphism.\qed
  \end{cor}

  The factors in the corollary describe the numbers in (ST-3$'$).

  \begin{lem}
    \label{lemma-num-entries}
    For an embedding $A\subset B$ with  socle tableau $\Sigma$
    and for natural numbers $\ell,r \geq 1$,
    $$\len\left(
    \frac{\soc^\ell A\cap\rad^{r-1}B}{\soc^{\ell-1}A\cap\rad^{r-1}B}\right)
    \,=\,
    \#\{\text{entries $\ell$ with row numbers $\geq r$ in $\Sigma$}\}.$$
  \end{lem}

  \begin{proof}
    For $S_\ell=B/\soc^\ell A$, note that $\len S_\ell$ counts the number
    of boxes in $\Sigma$ which are either empty or contain an entry $>\ell$.
    Hence, $\len\rad^{r-1}S_\ell$
    counts the number of boxes in $\Sigma$ which occur in rows $\geq r$
    and which are either empty or contain an entry $>\ell$.
    Thus in the lemma, the number on the right hand side is
    $$\len\rad^{r-1}S_{\ell-1}-\len\rad^{r-1}S_{\ell}.$$
    Note that
    $$\rad^{r-1}S_\ell=\rad^{r-1}\big(\frac B{\soc^\ell A}\big)=
    \frac{\rad^{r-1}B+\soc^\ell A}{\soc^\ell A}\cong
    \frac{\rad^{r-1}B}{\soc^\ell A\cap \rad^{r-1}B}.$$
    The formula implies that
    $$\len\rad^{r-1}S_{\ell-1}-\len\rad^{r-1}S_{\ell}=
    \len\left(\frac{\soc^\ell A\cap\rad^{r-1}B}{\soc^{\ell-1}A\cap\rad^{r-1}B}\right).$$
    This finishes the proof of the lemma.
  \end{proof}
  
  \begin{prop}
    The socle tableau $\Sigma$ of an embedding $A\subset B$ satisfies
    conditions {\rm (ST-1)} through {\rm (ST-3)}.
  \end{prop}

  \begin{proof}
    We have seen that conditions (ST-1) and (ST-2) hold for $\Sigma$.
    Here we verify property {\rm(ST-3$'$)} which is equivalent
    to (ST-3) by Lemma~\ref{lem-lpp-equiv}.

    \smallskip
    The map $\mu_p$ in Corollary~\ref{cor-mono} is a monomorphism, hence
    $$\len\left(\frac{\soc^\ell A\cap\rad^{s-1}B}{\soc^{\ell-1}A\cap\rad^{s-1}B}\right)
    \;\leq\;
    \len\left(\frac{\soc^{\ell-1}A\cap\rad^sB}{\soc^{\ell-2}A\cap\rad^sB}\right).$$
    The claim follows from Lemma~\ref{lemma-num-entries}.
  \end{proof}

\subsection{Every socle tableau is the socle tableau of an embedding.}

We will use the following tool to construct, for a given socle tableau $\Sigma$,
a corresponding embedding $(A\subset B)$.

\begin{lem}\label{lem-epi-sequence}
 Let 
 $$
 B=C_0\stackrel{f_1}\longrightarrow C_1\stackrel{f_2}\longrightarrow
 \cdots C_{s-1}\stackrel{f_s}\longrightarrow C_s=C
 $$
 be a~sequence of surjective $\Lambda$-homomorphisms
 with semisimple kernels satisfying the following condition
 $$(*)\qquad \soc(\Ker f_{\ell+1}f_\ell)=\Ker f_\ell\quad\text{for all}
 \quad \ell=1,\ldots, s-1.$$
 Then $A=\Ker f$ where $f=f_s\cdots f_1$
 has  partition type $\alpha$ given by
 $\alpha'_\ell=\len \Ker f_\ell$ for all $\ell$.
 Moreover, $B/\soc^\ell A\cong C_\ell$ holds for each $0\leq \ell\leq s$.
\end{lem}

\begin{proof}
  We first show by induction on $j$ that 
  for each $\ell$ the equality $\soc( \Ker f_{\ell+j}f_{\ell+j-1}\cdots f_\ell)=\Ker f_\ell$ holds.
  The case where $j=0,1$ is satisfied by  assumption.  Let $j>1$.
  The inclusion $\Ker f_\ell\subset \soc(\Ker f_{\ell+j}\cdots f_\ell)$ is clear, so
  let $x\in \soc (\Ker f_{\ell+j}\cdots f_\ell)$.
  If $x\in \Ker f_{\ell+j-1}\cdots f_\ell$, we are finished by  induction.
  If  $f_{\ell+j-1}\ldots f_\ell(x)\neq 0$, then 
  $f_\ell(x)\in \soc (\Ker f_{\ell+j}\cdots f_{\ell+1})$.
  By  induction, $f_\ell(x)\in\Ker f_{\ell+1}$ and we have
  a~contradiction to the assumption that $f_{\ell+j-1}\cdots f_\ell(x)\neq 0$.
   Hence $x\in\Ker f_{\ell}$.

  \smallskip
  Next we observe that for any $\ell=1,\ldots, s-1$ there is the following commutative diagram
  with exact rows and columns.
  $$
  \xymatrix{ &0\ar[d]&0\ar[d]&&\\
    &\Ker f_{\ell}\ar[d]\ar@{=}[r]&\Ker f_{\ell}\ar[d]&&\\
    0\ar[r]&\Ker f_s\cdots f_{\ell}\ar[r]\ar[d]&C_{\ell-1}\ar[r]^{f_s\cdots f_{\ell}}\ar[d]_{f_\ell}&C_s\ar[r]\ar@{=}[d]&0\\
    0\ar[r]&\Ker f_s\cdots f_{\ell+1}\ar[r]\ar[d]&C_\ell\ar[r]_{f_s\cdots f_{\ell+1}}\ar[d]&C_s\ar[r]&0\\
 &0&0&& 
 }
  $$

  In the first part we have seen that $\Ker f_\ell=\soc \Ker f_s\cdots f_{\ell}$;
  using the exactness of the left column in the diagram we obtain
  $$\frac{\Ker f_s\cdots f_{\ell}}{\soc (\Ker f_s\cdots f_{\ell})}\cong\Ker f_s\cdots f_{\ell+1}.$$

  Thus the Young diagram for $\Ker f_s\cdots f_{\ell}$ is obtained from the Young diagram
  for $\Ker f_s\cdots f_{\ell+1}$ by adding a new top row of length
  $$\len \soc(\Ker f_s\cdots f_\ell)=\len\Ker f_\ell.$$

  Suppose that $A=\Ker f_s\cdots f_1$ has partition type $\alpha$.
  Then it follows that $\alpha'_\ell = \len\Ker f_\ell$.

  \medskip
  Next we show by induction on $\ell$  that $B/\soc^\ell A\cong C_\ell$.
  More precisely, we show in each step that $\soc^\ell A=\Ker f_\ell\cdots f_1$.
  Since the epimorphism
  $$\pi_\ell:B\to C_\ell,\quad b\mapsto f_\ell\cdots f_1(b),$$
  maps $A$ onto $\Ker f_s\cdots f_{\ell+1}$, it has kernel $\soc^\ell A$ and hence
  induces the desired isomorphism $B/\soc^\ell A\cong C_\ell$.

  \smallskip
  For $\ell=0$ there is nothing to show.  Assume $\ell>0$.
  \begin{eqnarray*}
    \soc^\ell A     &=&\pi_{\ell-1}^{-1}(\soc\Ker f_s\cdots f_{\ell})\\
    &=& \pi_{\ell-1}^{-1}(\Ker f_{\ell})\\
    & = & \{b:\pi_{\ell-1}(b)\in\Ker f_{\ell}\}\\
    &=& \Ker f_\ell\cdots f_1\\
    &=&\Ker\pi_\ell
  \end{eqnarray*}
  The first equality holds by induction hypothesis:  $\pi_{\ell-1}:A\to \Ker f_s\cdots f_\ell$
  is an epimorphism with kernel $\soc^{\ell-1}A$.  Hence by definition of the $\ell$-th socle,
  $\soc^\ell A$ is the inverse image of the socle of $\Ker f_s\cdots f_\ell$.
  The equations show that $\pi_\ell$ induces an isomorphism $B/\soc^\ell A\cong C_\ell$.
\end{proof}

\begin{thm}
  \label{thm-kleinST}
  Every socle tableau can be realized as the socle
  tableau of an embedding.
\end{thm}

\begin{proof}  We first present an overview of the proof, then an example, and finally
  the general construction.

  \smallskip
  \boldit{Overview:} 
  Suppose that the socle tableau $\Sigma$ has shape $(\alpha,\beta,\gamma)$.
  We take $B=N_\beta$ and $C=N_\gamma$ and use Lemma~\ref{lem-epi-sequence}
  to construct an epimorphism $f:B\to C$ with kernel $A$ such that the 
  embedding $(A\subset B)$ has the desired socle tableau.

  \medskip
  The tableau $\Sigma$ is given by partitions $\sigma^{(\ell)}$,
  where $0\leq \ell\leq s$ with $s=\alpha_1$.  Let $C^{(\ell)}$ be a module of type $\sigma^{(\ell)}$,
  more precisely, put
  $$C^{(\ell)}=\textstyle \bigoplus_j C^{(\ell)}_j,$$
  where $C^{(\ell)}_j=P^{\sigma^{(\ell)}_j}$ is the indecomposable $\Lambda$-module of
  length $\sigma^{(\ell)}_j$.  It follows from (ST-1) and (ST-2) that 
  $\sigma^{(\ell)}\subset\sigma^{(\ell-1)}$, hence the $j$-th parts satisfy $\sigma^{(\ell)}_j\leq
  \sigma^{(\ell-1)}_j$, so there are canonical maps
  $g^{(\ell)}_j:C^{(\ell-1)}_j\to C^{(\ell)}_j$.  We can define the map $g^{(\ell)}:C^{(\ell-1)}\to C^{(\ell)}$
  as the diagonal map
  $$\textstyle g^{(\ell)}=\bigoplus_j g^{(\ell)}_j:
  \quad \bigoplus_j C^{(\ell-1)}_j\to \bigoplus_j C^{(\ell)}_j.$$
  (In case the partition $\sigma^{(\ell-1)}$ has more parts than $\sigma^{(\ell)}$,
  formally add parts of size 0.)
  Since $\sigma^{(\ell-1)}\setminus\sigma^{(\ell)}$ is a horizontal strip of length $\alpha'_\ell$,
  the kernel of $g_\ell$ is a semisimple $\Lambda$-module of length $\alpha'_\ell$.

  \medskip
  For $\ell=1,\ldots,s-1$ we construct automorphisms $h^{(\ell)}:C^{(\ell)}\to C^{(\ell)}$
  such that $\soc (\Ker g^{(\ell+1)} h^{(\ell)} g^{(\ell)})=\Ker g^{(\ell)}$.
  Then the maps $$f_\ell=\left\{\begin{array}{ll}h^{(\ell)} g^{(\ell)}, &\text{if}\; 1\leq\ell<s\\
  g^{(s)} & \text{if}\;\ell=s\end{array}\right.$$ satisfy Condition (*) in
  Lemma~\ref{lem-epi-sequence}.

  \smallskip
  \boldit{An example:}
  We briefly pause the proof to illustrate the situation in the example of
  socle tableau $\Sigma_2$ from above.
  $$  \Sigma_2:\raisebox{-5mm}{\STtwo} $$
  Consider the first column.  Since it contains a box \singlebox1, the map
  $g^{(1)}_1:P^5\to P^4$ has a one-dimensional kernel.  Similarly for the box
  \singlebox2, the map $g^{(2)}_1:P^4\to P^3$ has a one-dimensional kernel.
  The kernel of the composition $g^{(2)}_1\circ g^{(1)}_1$ is isomorphic to
  $\Lambda/(p^2)$, hence the condition in Lemma~\ref{lem-epi-sequence} is satisfied:
  $$\soc(\Ker g^{(2)}_1\circ g^{(1)}_1)=\Ker g^{(1)}_1$$
  Going on, since there are no boxes \singlebox 3\ or \singlebox 4\ in the  first
  column, we have $g^{(3)}_1=1_{P^3}=g^{(4)}_1$.  Hence for $\ell=2,3$, the condition
  $\soc(\Ker g^{(\ell+1)}_1\circ g^{(\ell)}_1)=\Ker g^{(\ell)}_1$ is satisfied.

  \smallskip
  We see that the only situation where the condition is violated occurs when
  $g^{(\ell+1)}_j$ is a proper epimorphism and $g^{(\ell)}_j$ the identity map, that is,
  when the $j$-th column in $\Sigma$ contains a box $b$ with entry
  $\ell+1$ but no box with entry $\ell$.  In our example, this occurs three times,
  when $(\ell,j)\in\{(1,3), (2,2), (3,3)\}$.
  Let us consider the case $(\ell,j)=(2,2)$, where we have a box \singlebox3\ in the
  second column for which the corresponding box \singlebox2\ occurs
  in the first column. 
  We treat both columns simulaneously to satisfy the condition in the Lemma,
  this involves the maps $g^{(2)}_1\oplus g^{(2)}_2$,
  $h^{(2)}_{1,2}=\left(\begin{smallmatrix} 1 & \text{incl}\\
    0 & 1\end{smallmatrix}\right)$,
  and $g^{(3)}_1\oplus g^{(3)}_2$ where
  $\incl=\mu_p:P^2\to P^3, a\mapsto pa$, is the inclusion map:
  $$\xymatrix{P^4\oplus P^2\ar[r]^{\text{can}\oplus1}
    & P^3\oplus P^2\ar[r]^{\left(\begin{smallmatrix} 1 & \text{incl}\\
        0 & 1\end{smallmatrix}\right)} & P^3\oplus P^2\ar[r]^{1\oplus\text{can}} & P^3\oplus P^1}$$
                The composition $C$ of the three maps is the
  epimorphism in the short exact sequence
  $$\xymatrix{0\ar[r] & P^2\ar[r]^{\left(\begin{smallmatrix}\mu_{p^2}\\
        -\mu_p \end{smallmatrix}\right)}
    & \quad P^4\oplus P^2\quad
    \ar[r]^{\left(\begin{smallmatrix}\text{can} & \text{incl}\\
        0 & \text{can}\end{smallmatrix}\right)}
    & \;P^3\oplus P^1\ar[r] & 0}$$
  We see that $\soc \Ker C=\soc P^4=\Ker g^{(2)}_1\oplus g^{(2)}_2$,
  so Condition (*) in Lemma~\ref{lem-epi-sequence} is satisfied.

  \smallskip
  \boldit{The general construction:}
  For each $\ell=1,\cdots,s-1$ we construct the automorphism $h^{(\ell)}$ of $C^{(\ell)}$.
  We use condition (ST-3$''$) on $\Sigma$ to partition the set of columns of
  the Young diagram for $\sigma^{(\ell)}$
  into subsets which consist either of a single column or a pair of columns.
  Using the map $\varphi^{(\ell)}$, we consider the following cases:

  \smallskip
  (a)  Column $j$ contains boxes with entries $\ell$ and $\ell+1$; by definition
  of $\varphi^{(\ell)}$, the box with entry $\ell+1$ is mapped to the box with entry $\ell$.
  In this case, we add the singleton $\{j\}$ to our partition of the columns and
  put $h^{(\ell)}_j=1$, the identity map on $C^{(\ell)}_j$.

  \smallskip
  (b) Column $j$ contains neither a box with entry $\ell$ nor one with entry $\ell+1$.
  Again, column $j$ forms a singleton, and we put $h^{(\ell)}_j=1$.

  \smallskip
  (c) Column $j$ contains a box with entry $\ell$ which is not in the image of
  $\varphi^{(\ell)}$.  Also in this case, column $j$ is a singleton and $h^{(\ell)}_j=1$.

  \smallskip
  (d) The remaining columns contain either a box with entry $\ell+1$, or a box
  with entry $\ell$ in the image of $\varphi^{(\ell)}$.  Suppose column $i$ contains
  entry $\ell$ and column $j$ entry $\ell+1$ and $\varphi^{(\ell)}$ maps the box in
  column $j$ to the box in column $i$.  Then $i<j$ and if
  $u=\sigma^{(\ell)}_i$, $v=\sigma^{(\ell)}_j$ are the lengths of the two columns
  in the Young diagram for $\sigma^{(\ell)}$, then $u>v$.
  In this case, we add the pair $\{i,j\}$ to our partition of the set of columns,
  and define
  $$h^{(\ell)}_{i,j}:C^{(\ell)}_i\oplus C^{(\ell)}_j\to C^{(\ell)}_i\oplus C^{(\ell)}_j, \quad
  {x \choose y}\mapsto \left(\begin{array}{cc} 1 & \incl\\ 0 & 1\end{array}\right) \cdot{x\choose y}$$
  where $\incl$ is the inclusion map from $C^{(\ell)}_j=P^v$ to $C^{(\ell)}_i=P^u$
  given by multiplication by $p^{u-v}$.

  \smallskip
  We have accounted for each column.
  Then $h^{(\ell)}=\bigoplus_S h^{(\ell)}_S$ where
  $S$ runs over the parts of the partition of the set of columns is an isomorphism
  of $C^{(\ell)}$.  Put $g^{(\ell)}_S=\bigoplus_{j\in S}\,g^{(\ell)}_j$, then one can check that
  in each of the four cases (a) -- (d),
  $$\soc(\Ker g^{(\ell+1)}_S\, h^{(\ell)}_S\, g^{(\ell)}_S)=\Ker g^{(\ell)}_S$$
  holds (for case (d), see the example in this proof).
  Putting $f_\ell=h^{(\ell)}\,g^{(\ell)}$ (where $h^{(s)}=1$),
  we see that Condition (*) in Lemma~\ref{lem-epi-sequence} is satisfied.
  Let $A=\Ker f_s\cdots f_1$, then
  the embedding $(A\subset B)$ has socle tableau $\Sigma$.
\end{proof}

\section{Socle tableau, dual LR-tableau and Hom-matrix}
\label{sec-soc-lr-hom}

Let $X$ be an embedding of type $(\alpha,\beta,\gamma)$.
The \boldit{Hom-matrix} for $X$ describes the sizes of the homomorphism
groups from the pickets into $X$.  The matrix $H=(h_\ell^m)_{\ell,m}$
is given by $h_\ell^m=\len\Hom_{\mathcal S}(P_\ell^m,X)$.

\medskip
In this section, we discuss the interplay between the socle tableau
for $X$, the LR-tableau of the dual embedding $DX$ and the Hom-matrix $H$.

\begin{thm}
  \label{thm-lr-soc-hom}
  For an embedding $X$, the following
  invariants are equivalent in the sense that each one determines both
  of the others.
  \begin{enumerate}
  \item The socle tableau $\Sigma=\Sigma_X$.
  \item The LR-tableau of the dual embedding $\Gamma^*=\Gamma_{DX}$.
  \item The Hom-matrix $H=(h_\ell^m)_{\ell,m}$ where
    $h_\ell^m= \hom_{\mathcal S}(P_\ell^m,X)$.
  \end{enumerate}
\end{thm}

In the following three subsections, we give explicit combinatorial
formulas   for the conversion 
between the socle tableau and the Hom-matrix; the dual LR-tableau
and the Hom-matrix; and the socle tableau and the dual LR-tableau.

\subsection{The socle tableau and the Hom-matrix}
\label{subsec-one}
We denote by $\mu_\Sigma(\boxedentry3er)$ the multiplicity of the entry $e$
in row $r$ in $\Sigma$.

\medskip
Our aim is to show the two formulas:
\begin{eqnarray*}
  \mu_\Sigma(\boxedentry3\ell r) &= &
  h_\ell^{r+\ell-1}-h_\ell^{r+\ell}-h_{\ell-1}^{r+\ell-2}+h_{\ell-1}^{r+\ell-1}\\
  & = & h_\ell^{m-1}-h_\ell^m-h_{\ell-1}^{m-2}+h_{\ell-1}^{m-1}
\end{eqnarray*}
where we substitute $m=r+\ell$ in the second line.
Conversely, one can retrieve each entry in the Hom-matrix from the
socle tableau.
\begin{eqnarray*}h_\ell^m & = &
  |\soc^\ell A| + |\soc^{m-\ell}(B/\soc^\ell A)|\\
  & = & \alpha'_1+\cdots+\alpha'_\ell+
  (\sigma^{(\ell)})'_1+\cdots+(\sigma^{(\ell)})'_{m-\ell}
\end{eqnarray*}

\bigskip
For the proof of the first formula, we use that for a $\Lambda$-module $U$
of type $\lambda$, the length of the $r$-th row of the Young diagram
for $\lambda$ is
$$\lambda'_r=\len\soc^rU-\len\soc^{r-1}U.$$
We also use the following

\begin{lem}
  For an embedding $(A\subset B)$ and a picket $P_\ell^m$ with $m=\ell+r$
  we have:
  $$\hom_{\mathcal S}(P_\ell^m,(A\subset B))= 
  \hom_\Lambda(P^\ell,A) + \hom_\Lambda(P^r,B/\soc^\ell A)$$
\end{lem}

\begin{proof}
  \begin{eqnarray*}
    \lefteqn{\hom_{\mathcal S}(P_\ell^m,(A\subset B)) = }\\
    \qquad & =& \hom_{\mathcal S}(P_\ell^m,(\soc^\ell A\subset B))\\
    &=& \hom_\Lambda(P^m,p_B^{-r}(\soc^\ell A))\\
    &=& \hom_\Lambda(P^m,\soc^\ell A)
    +\hom_\Lambda(P^m,p_B^{-r}(\soc^\ell A)/\soc^\ell A)\\
    &=& \hom_\Lambda(P^\ell,A) + \hom_\Lambda(P^r,B/\soc^\ell A)
  \end{eqnarray*}
        For the first two equalities note that
        each $f\in\Hom_{\mathcal S}(P_\ell^m,(A\subset B))$ is given by an element
        $b\in B$ such that $p^rb\in A$ and $p^{r+\ell}b=0$.
        Since $p^{-r}_B(\soc ^\ell A)$ is a $\Lambda/(p^m)$-module,
        the next equality follows from the exactness of the Hom-functor.
        For the last step note that each $f\in\Hom_\Lambda(P^s,C)$
        is given by an element in $p^{-s}_C(0)$.
\end{proof}

We can now show the first equality.
\begin{eqnarray*}
  \mu_\Sigma(\boxedentry3\ell r) & = &
  (\sigma^{(\ell-1)})'_r-(\sigma^{(\ell)})'_r\\
  &=& \len\soc^r(B/\soc^{\ell-1}A) -\len\soc^{r-1}(B/\soc^{\ell-1}A)\\
  && {}-\len\soc^r(B/\soc^\ell A) +\len\soc^{r-1}(B/\soc^\ell A)\\
  &=& \hom(P^r,B/\soc^{\ell-1}A) -\hom(P^{r-1},B/\soc^{\ell-1}A)\\
  && {}-\hom(P^r,B/\soc^\ell A) +\hom(P^{r-1},B/\soc^\ell A)\\
  &=& \hom(P^r,B/\soc^{\ell-1}A)+\hom(P^{\ell-1},A)\\
  && {}-\hom(P^{r-1},B/\soc^{\ell-1}A)-\hom(P^{\ell-1},A)\\
  &&{}-\hom(P^r,B/\soc^\ell A)-\hom(P^\ell,A)\\
  && {}+\hom(P^{r-1},B/\soc^\ell A)+\hom(P^\ell,A)\\
  &=& \hom_{\mathcal S}(P^{m-1}_{\ell-1},(A\subset B))
  - \hom_{\mathcal S}(P^{m-2}_{\ell-1},(A\subset B))\\
  &&{} - \hom_{\mathcal S}(P^{m}_{\ell},(A\subset B))
  + \hom_{\mathcal S}(P^{m-1}_{\ell},(A\subset B))\\
  &=& h_{\ell-1}^{m-1}-h_{\ell-1}^{m-2}-h_{\ell}^m+h_\ell^{m-1}
\end{eqnarray*}

\begin{cor}
  \label{cor-cok-hom}
  Let $\ell,r$ be natural numbers, $m=\ell+r$ and $f_\ell^m$ the
  monomorphism in the short exact sequence (where we put $P_0^0=0$).
  $$0\longrightarrow P_\ell^{m-1}\stackrel{f_\ell^m}\longrightarrow
  P_\ell^m\oplus P_{\ell-1}^{m-2}\longrightarrow
  P_{\ell-1}^{m-1}\longrightarrow 0$$
  If $\Sigma$ is the socle tableau of the embedding $X\in\mathcal S(\Lambda)$,
  then
  $$\mu_\Sigma(\boxedentry 3\ell r)=\len\Cok\Hom_{\mathcal S}(f_\ell^m,X).$$
\end{cor}

  \begin{rem}
    In the case where $\Lambda$ is bounded, for example if $\Lambda=k[T]/(T^n)$
    or $\Lambda=\mathbb Z/(p^n)$, and $n=m-1$
    then the summand $P_\ell^m$ in the
    middle term of the exact sequence in Corollary~\ref{cor-cok-hom}
    is not defined.  In this case, 
    $$\mu_\Sigma(\boxedentry 3\ell r)=\len\Cok\Hom_{\mathcal S}(f_\ell^m,X)$$
    where $f_\ell^m$ is the canonical map $P_\ell^{m-1}\to P_{\ell-1}^{m-2}$.
  \end{rem}

\bigskip
The second formula is shown by the following equations.
\begin{eqnarray*}
  h_\ell^m &=& \hom_{\mathcal S}(P_\ell^m,(A\subset B))\\
  &=& \hom(P_\ell^m, (\soc^\ell A\subset B))\\
  &=& \hom(P_\ell^m, (\soc^\ell A\subset p^{-(m-\ell)}_B(\soc^\ell A)))\\
  &=& \hom_\Lambda(P^m,p_B^{-(m-\ell)}(\soc^\ell A))\\
  &=& \len p_B^{-(m-\ell)}(\soc^\ell A)\\
  &=& \len \soc^\ell A + \len \soc^{m-\ell}(B/\soc^\ell A)\\
  &=& \alpha'_1+\cdots+\alpha'_\ell\;+\;
  (\sigma^{(\ell)})'_1+\cdots+(\sigma^{(\ell)})'_{m-\ell}\\
\end{eqnarray*}

\subsection{The dual LR-tableau and the Hom-matrix}

Our aim in this subsection is to verify two formulas.
$$\mu_{\Gamma^*}(\boxedentry 3\ell m)=
\left\{\begin{array}{ll}
h_r^m-h_{r+1}^m-h_{r-1}^{m-1}+h_r^{m-1} & \text{if}\; \ell<m \\
h_0^m - h_1^m & \text{if}\; \ell=m
\end{array}\right.
$$

\medskip
We write $m=r+\ell$ as above.  Conversely, we can obtain $h_r^m$
from the tableau $\Gamma^*=(\lambda^{(\ell)})_\ell$ via
$$h_r^m=(\lambda^{(\ell)})'_1+\cdots+(\lambda^{(\ell)})'_m.$$

\bigskip
The first formula follows from \cite[Theorem~2]{s-lrtableau}:
$$\mu_{\Gamma^*}(\boxedentry3\ell m)=\len\Cok\Hom_{\mathcal S}(\tilde h^m_r,X)$$
where $\tilde h^m_r$ is the map
$$\tilde h^m_r:\left\{\begin{array}{ll}
P_0^m\to P_1^m, & \text{if}\;r=0\\
P_r^m\to P_{r-1}^{m-1}\oplus P_{r+1}^m, & \text{if}\; 1\leq r<m\\
P_m^m\to P_{m-1}^{m-1}, & \text{if}\; r=m.
\end{array}\right.
$$
In the case where $\ell<m$ we have $1\leq r<m$ and the formula is obtained
by applying the contravariant Hom-functor to the short exact sequence
$$0\longrightarrow P_r^m\stackrel{\tilde h_r^m}\longrightarrow
P_{r-1}^{m-1}\oplus P_{r+1}^m\longrightarrow P_r^{m-1}\longrightarrow 0.$$
If $\ell =m$ we have $r=0$.  Here we can use the short exact sequence
$$0\longrightarrow\Hom(P_1^m,X)\longrightarrow\Hom(P_0^m,X)
\longrightarrow\Cok\Hom(\tilde h_0^m,X)\longrightarrow 0.$$

\bigskip
For the proof of the second formula we use

\begin{lem}\label{hom-dual}
  Suppose the embedding $X=(A\subset B)$
  has LR-tableau $\Gamma=(\gamma^{(i)})_i$ (so $\gamma^{(i)}$ is the type
  of $B/p^iA$).
  Then
  $$\hom_{\mathcal S}(Y,P_\ell^m)=(\gamma^{(\ell)})'_1+\cdots+(\gamma^{(\ell)})'_m.$$
\end{lem}

\begin{proof}
  With the notation from the Lemma we have:
  \begin{eqnarray*}
    \hom_{\mathcal S}(X,P_\ell^m) & = &
    \hom_{\mathcal S}((A/p^\ell A\subset B/p^\ell A),P_\ell^m)\\
    &=& \hom_\Lambda(B/p^\ell A,P^m)\\
    &=& (\gamma^{(\ell)})'_1+\cdots+(\gamma^{(\ell)})'_m.
  \end{eqnarray*}
\end{proof}

We can now show the second formula:
\begin{lem}
  \label{lemma-dual-lr}
  Suppose the embedding $X=(A\subset B)$ has dual LR-tableau
  $\Gamma^*=(\lambda^{(\ell)})_\ell$.  Let $m$ be a natural number,
  $0\leq r\leq m$ an integer and $\ell=m-r$. Then
  $$h_r^m=(\lambda^{(\ell)})'_1+\cdots+(\lambda^{(\ell)})'_m.$$
\end{lem}

\begin{proof}
  \begin{eqnarray*}
    h_r^m & = & \hom_{\mathcal S}(P_r^m,X)\\
    &=& \hom_{\mathcal S}(DX,DP_r^m)\\
    &=& \hom(DX,P_{m-r}^m)\\
    &=& (\lambda^{(m-r)})'_1+\cdots+(\lambda^{(m-r)})'_m
  \end{eqnarray*}
  where the last step follows from Lemma~\ref{hom-dual}.
\end{proof}

\subsection{The socle tableau and the dual LR-tableau}
\label{sec-lr-soc}
We show two formulas which show how to obtain the socle tableau
from the dual LR-tableau and conversely.  Of course, this can be done
via the Hom-matrix.

\begin{prop}
  Suppose an embedding has socle tableau $\Sigma$
  and the dual embedding has  LR-tableau
  $\Gamma^*=(\lambda^{(i)})_i$.  If $\ell,r$ are natural numbers and $m=\ell+r$,
  we have
  $$\mu_\Sigma(\boxedentry 3\ell r) = (\lambda^{(r-1)})'_{m-1}-(\lambda^{(r)})'_m.$$
\end{prop}

\begin{proof}
The formula follows from the first statement in subsection~\ref{subsec-one}
and Lemma~\ref{lemma-dual-lr}:
\begin{eqnarray*}
  \mu_\Sigma(\boxedentry 3\ell r)& = &
  h_\ell^{m-1}-h_\ell^m-h_{\ell-1}^{m-2}+h_{\ell-1}^{m-1}\\
  & = & (\lambda^{(r-1)})'_1+\ldots+(\lambda^{(r-1)})'_{m-1}
  - (\lambda^{(r)})'_1-\ldots-(\lambda^{(r)})'_m\\
  && {} - (\lambda^{(r-1)})'_1-\ldots-(\lambda^{(r-1)})_{m-2}
  {} + (\lambda^{(r)})'_1+\ldots+(\lambda^{(r)})'_{m-1}\\
  & = & (\lambda^{(r-1)})'_{m-1}-(\lambda^{(r)})'_m
\end{eqnarray*}
\end{proof}

The second formula summarizes how to retrieve the entries in $\Gamma^*$
from $\Sigma$.
$$\mu_{\Gamma^*}(\boxedentry3\ell m) =
\left\{\begin{array}{ll}
0 & \text{if} \; m<\ell \\
\displaystyle \beta'_m-\sum_{j\geq m}\mu_\Sigma(\boxedentry31j)
& \text{if}\; m=\ell\\
\displaystyle \sum_{j>\ell}\mu_\Sigma(\boxedentry7{m-\ell}j)
- \sum_{j\geq\ell}\mu_\Sigma(\boxedentry{11}{m-\ell-1}j)
& \text{if} \; m>\ell
\end{array}\right.
$$

Note that the first sum in the last line counts the number of entries
$r=m-\ell$ underneath the $\ell$-th row, while the second sum counts
the entries $m-\ell+1$ in the $\ell$-th row and underneath.
We observe that by (ST-3$'$) the difference is nonnegative.

\medskip
The formula follows from the rule
$\mu_{\Gamma^*}(\boxedentry3\ell m) = (\lambda^{(\ell)})'_m-(\lambda^{(\ell-1)})'_m$
and the following

\begin{lem}
  $$(\lambda^{(\ell)})'_m=\left\{
  \begin{array}{ll}
    \beta'_m & \text{if}\; m\leq \ell \\
    \sum_{j>\ell} \mu_\Sigma(\boxedentry7{m-\ell}j) & \text{if}\; m>\ell
  \end{array}
  \right.
  $$
\end{lem}

\begin{proof}
  Note that the first $\ell$ rows of $\beta$
  and $\lambda^{(\ell)}$ are equal (because there is no entry
  bigger than $\ell$ in the first $\ell$ rows of an~LR-tableau), hence 
  $(\lambda^{(\ell)})'_m=\beta_m'$ for all $m\leq \ell$.
  Now, we use the first formula in this subsection repeatedly:
  \begin{eqnarray*}
    (\lambda^{(\ell)})'_m & = &
    (\lambda^{(\ell-1)})'_{m-1}-\mu_\Sigma(\boxedentry7{m-\ell}\ell)\\
    & = & (\lambda^{(\ell-2)})'_{m-2}-\mu_\Sigma(\boxedentry7{m-\ell}{\ell-1}\quad)
    -\mu_\Sigma(\boxedentry7{m-\ell}\ell)\\
    & = & \cdots \\
    & = & (\lambda^{(0)})'_{m-\ell}
    -\textstyle\sum_{j=1}^\ell \mu_\Sigma(\boxedentry7{m-\ell}j)\\
    & = & \alpha'_{m-\ell}
    -\textstyle\sum_{j=1}^\ell \mu_\Sigma(\boxedentry7{m-\ell}j)\\
    & = & \textstyle\sum_{j>\ell} \mu_\Sigma(\boxedentry7{m-\ell}j)
  \end{eqnarray*}
\end{proof}

As a~consequence we obtain the following characterisation
of the Littlewood-Richardson coefficient $c_{\alpha,\gamma}^\beta$.

\begin{cor}\label{cor-LR-coeff}
  The number $c_{\alpha,\gamma}^\beta$ equals the number of socle tableaux
  of shape $(\alpha,\beta,\gamma)$. \qed
\end{cor}

\section{Applications and a~Conjecture}
\label{sec-applications}

  First, given an embedding in one of the categories
  $\mathcal S(n)$ (see below) and its socle tableau,
 we determine for
 each entry in the tableau the possible positions of the
 indecomposable direct summands of the embedding
 within the Auslander-Reiten quiver for $\mathcal S(n)$.

 \smallskip
 Next we consider the representation space for embeddings
 of type $(42,642,42)$, it is the
 constructible variety of short exact sequences
 $$0\longrightarrow N_\alpha\longrightarrow N_\beta\longrightarrow N_\gamma
 \longrightarrow 0$$
 where  $\alpha=(42),\beta=(642),\gamma=(42)$.
 This space contains the first occurance of a family
 of pairwise nonisomorphic indecomposable embeddings.
 We describe how socle tableaux partition this representation space.

 \smallskip
 Finally,
 we note that the correspondence between the socle tableau of an embedding
 and the dual LR-tableau is related
 to the tableau switching studied in \cite{bss}.
 We give an example and conclude with an open problem.

\subsection{Positioning objects in the Auslander-Reiten quiver}

Given an indecomposable embedding $M$ with socle tableau $\Sigma$,
we show that each entry $\ell$ in a given row $r$
in $\Sigma$ determines a region in the
Auslander-Reiten quiver $\Gamma$ in which $M$ does occur.
This region is given by those embeddings $X$ for which the
contravariant defect $\mu_\Sigma(\boxedentry 3\ell r)$
in the formula in Corollary~\ref{cor-cok-hom}
is positive; equivalently, the region is given by those objects $X$
for which there exists a homomorphism from $P_\ell^{\ell+r-1}$ to $X$
which does not factor through $f_\ell^{\ell+r}$.
  
  \begin{figure}[ht]
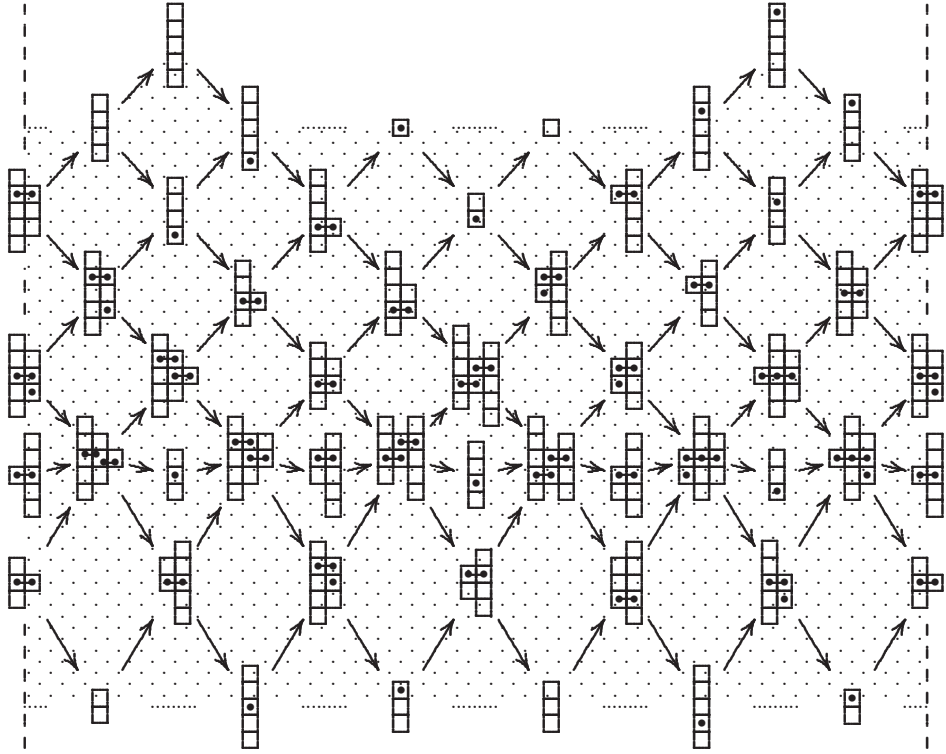

    \caption{The Auslander-Reiten quiver for $\mathcal S(5)$}
    \label{fig-arq}
$$
\hbox{\beginpicture
\setcoordinatesystem units <1cm,1.1cm>
\multiput{\ssq} at 1.9 -.1  1.9 .1 /  
\multiput{\ssq} at 3.9 -.4  3.9 -.2  3.9 0  3.9 .2  3.9 .4 /
\put{\bul} at 4 0
\multiput{\ssq} at 5.9 -.2  5.9 0  5.9 .2 /
\put{\bul} at 6 .2
\multiput{\ssq} at 7.9 -.2  7.9 0  7.9 .2 /
\multiput{\ssq} at 9.9 -.4  9.9 -.2  9.9 0  9.9 .2  9.9 .4 /
\put{\bul} at 10 -.2  
\multiput{\ssq} at 11.9 -.1  11.9 .1 /
\put{\bul} at 12 .1
\multiput{\ssq} at .8 1.3  .8 1.5  .8 1.7  1 1.5 
                3 1.1  3 1.3  3 1.5  3 1.7  3 1.9  2.8 1.5  2.8 1.7 
                4.8 1.1  4.8 1.3  4.8 1.5  4.8 1.7  4.8 1.9  5 1.3  5 1.5  
                                5 1.7
                7 1.2  7 1.4  7 1.6  7 1.8  6.8 1.4  6.8 1.6
                9 1.1  9 1.3  9 1.5  9 1.7  9 1.9  8.8 1.3  8.8 1.5  8.8 1.7
                10.8 1.1  10.8 1.3  10.8 1.5  10.8 1.7  10.8 1.9  
                                11 1.3  11 1.5
                                12.8 1.3  12.8 1.5  12.8 1.7  13 1.5 /
\multiput{\bul} at .9 1.5  1.1 1.5
                2.9 1.5  3.1 1.5  4.9 1.7  5.1 1.7  5.1 1.5
              6.9 1.6  7.1 1.6  8.9 1.3  9.1 1.3  10.9 1.5  11.1 1.5  
                                11.1 1.3
                                12.9 1.5  13.1 1.5 /
\plot .9 1.5  1.1 1.5 /
\plot 2.9 1.5  3.1 1.5 /
\plot 4.9 1.7  5.1 1.7 /
\plot 6.9 1.6  7.1 1.6 /
\plot 8.9 1.3  9.1 1.3 /
\plot 10.9 1.5  11.1 1.5 /
\plot 12.9 1.5  13.1 1.5 /
%
%
\multiput{\ssq} at 
1 2.4  1 2.6  1 2.8  1 3  1 3.2  .8 2.8
1.7 2.6  1.7 2.8  1.7 3  1.7 3.2  1.7 3.4  1.9 2.8  1.9 3  1.9 3.2  2.1 3
2.9 2.6  2.9 2.8  2.9 3
3.7 2.6  3.7 2.8  3.7 3  3.7 3.2  3.7 3.4  3.9 2.8  3.9 3  3.9 3.2  
                4.1 3  4.1 3.2
5 2.4  5 2.6  5 2.8  5 3  5 3.2  4.8 2.8  4.8 3
5.7 2.8  5.7 3  5.7 3.2  5.7 3.4  6.1 2.6  6.1 2.8  6.1 3  6.1 3.2  6.1 3.4
                5.9 3  5.9 3.2
6.9 2.5  6.9 2.7  6.9 2.9  6.9 3.1
7.7 2.6  7.7 2.8  7.7 3  7.7 3.2  7.7 3.4  8.1 2.6  8.1 2.8  8.1 3  8.1 3.2
                7.9 2.8  7.9 3
9 2.4  9 2.6  9 2.8  9 3  9 3.2  8.8 2.6  8.8 2.8 
9.9 2.6  9.9 2.8  9.9 3  9.9 3.2  9.9 3.4  10.1 2.8  10.1 3  10.1 3.2  9.7 2.8
               9.7 3 
10.9 2.6  10.9 2.8  10.9 3
11.9 2.6  11.9 2.8  11.9 3  11.9 3.2  11.9 3.4  12.1 2.8  12.1 3  12.1 3.2
               11.7 3
               13 2.4  13 2.6  13 2.8  13 3  13 3.2  12.8 2.8 /
\multiput{\bul} at 
                .9 2.8  1.1 2.8
                1.8 3.05  
        1.95 3.05
        2.05 2.95  2.2 2.95  3 2.8  3.8 3.2  4 3.2  4 3  4.2 3  4.9 3  5.1 3
        5.8 3  6 3  6 3.2  6.2 3.2  7 2.7  7.8 2.8  8 2.8  8 3  8.2 3
        8.9 2.8  9.1 2.8  9.8 2.8  9.8 3  10 3  10.2 3  11 2.6  11.8 3
        12 3  12.2 3  12.2 2.8
                        12.9 2.8  13.1 2.8 /
\plot .9 2.8  1.1 2.8 /
\plot 1.8 3.05  1.95 3.05 /
\plot 2.05 2.95  2.2 2.95 /
\plot 3.8 3.2  4 3.2 /
\plot 4 3  4.2 3 /
\plot 4.9 3  5.1 3 /
\plot 5.8 3  6 3 /
\plot 6 3.2  6.2 3.2 /
\plot 7.8 2.8  8 2.8 /
\plot 8 3  8.2 3 /
\plot 8.9 2.8  9.1 2.8 /
\plot 9.8 3  10.2 3 /
\plot 11.8 3  12.2 3 / 
\plot 12.9 2.8  13.1 2.8 /
\multiput{\ssq} at .8 3.6  .8 3.8  .8 4  .8 4.2  .8 4.4  1 3.8  1 4  1 4.2
        2.7 3.6  2.7 3.8  2.7 4  2.7 4.2  2.7 4.4  2.9 3.8  2.9 4  2.9 4.2
                                3.1 4
        4.8 3.7  4.8 3.9  4.8 4.1  4.8 4.3  5 3.9  5 4.1
        6.7 3.7  6.7 3.9  6.7 4.1  6.7 4.3  6.7 4.5  6.9 3.9  6.9 4.1  7.1 
                                3.5  7.1 3.7  7.1 3.9  7.1 4.1  7.1 4.3
        8.8 3.9  8.8 4.1  9 3.7  9 3.9  9 4.1  9 4.3
        10.7 4  10.9 3.6  10.9 3.8  10.9 4  10.9 4.2  10.9 4.4  11.1 3.8
                                11.1 4  11.1 4.2
        12.8 3.6  12.8 3.8  12.8 4  12.8 4.2  12.8 4.4  13 3.8  13 4  13 4.2 /
\multiput{\bul} at .9 4  1.1 4  1.1 3.8
                2.8 4.2  3 4.2  3 4  3.2 4  
        4.9 3.9  5.1 3.9  6.8 3.9  7 3.9  7 4.1  7.2 4.1  8.9 3.9  8.9 4.1 
        9.1 4.1  10.8 4  11 4  11.2 4
        12.9 4  13.1 4  13.1 3.8 /
\plot .9 4  1.1 4 /
\plot 2.8 4.2  3 4.2 /
\plot 3 4  3.2 4 /
\plot 4.9 3.9  5.1 3.9 /
\plot 6.8 3.9 7 3.9 /
\plot 7 4.1  7.2 4.1 /
\plot 8.9 4.1  9.1 4.1 /
\plot 10.8 4  11.2 4 /
\plot 12.9 4  13.1 4 /
\multiput{\ssq} at 
        1.8 4.6  1.8 4.8  1.8 5  1.8 5.2  1.8 5.4  2 4.8  2 5  2 5.2
        3.8 4.7  3.8 4.9  3.8 5.1  3.8 5.3  4 4.9
        5.8 4.6  5.8 4.8  5.8 5  5.8 5.2  5.8 5.4  6 4.8  6 5
        7.8 5  7.8 5.2  8 4.6  8 4.8  8 5  8 5.2  8 5.4
        9.8 5.1  10 4.7  10 4.9  10 5.1  10 5.3
        11.8 4.6  11.8 4.8  11.8 5  11.8 5.2  11.8 5.4  12 4.8  12 5  12 5.2 /
\multiput{\bul} at 
        1.9 5.2  2.1 5.2  2.1 4.8  3.9 4.9  4.1 4.9
        5.9 4.8  6.1 4.8  7.9 5  7.9 5.2  8.1 5.2  9.9 5.1  10.1 5.1  
        11.9 5  12.1 5 /
\plot 1.9 5.2  2.1 5.2 /
\plot 3.9 4.9  4.1 4.9 /
\plot 5.9 4.8  6.1 4.8 /
\plot 7.9 5.2  8.1 5.2 /
\plot 9.9 5.1  10.1 5.1 /
\plot 11.9 5  12.1 5 /
\multiput{\ssq} at .8 5.6  .8 5.8  .8 6  .8 6.2  .8 6.4  1 5.8  1 6  1 6.2
        2.9 5.7  2.9 5.9  2.9 6.1  2.9 6.3 
        4.8 5.6  4.8 5.8  4.8 6  4.8 6.2  4.8 6.4  5 5.8
        6.9 5.9  6.9 6.1
        8.8 6.2  9 5.6  9 5.8  9 6  9 6.2  9 6.4
        10.9 5.7  10.9 5.9  10.9 6.1  10.9 6.3
        12.8 5.6  12.8 5.8  12.8 6  12.8 6.2  12.8 6.4  13 5.8  13 6  13 6.2 /
\multiput{\bul} at .9 6.2  1.1 6.2
        3 5.7  4.9 5.8  5.1 5.8  7 5.9  8.9 6.2
        9.1 6.2  11 6.1 12.9 6.2  13.1 6.2 /
\plot .9 6.2  1.1 6.2 /
\plot 4.9 5.8  5.1 5.8 /
\plot 8.9 6.2  9.1 6.2 /
\plot 12.9 6.2  13.1 6.2 /
\multiput{\ssq} at 1.9 6.7  1.9 6.9
        1.9 7.1  1.9 7.3
        2.9 7.6  2.9 7.8  2.9 8  2.9 8.2  2.9 8.4  
        3.9 6.6  3.9 6.8  3.9 7  3.9 7.2  3.9 7.4
        5.9 7
        7.9 7
        9.9 6.6  9.9 6.8  9.9 7  9.9 7.2  9.9 7.4
        10.9 7.6  10.9 7.8  10.9 8  10.9 8.2  10.9 8.4
        11.9 6.7  11.9 6.9  11.9 7.1  11.9 7.3 /
\multiput{\bul} at 4 6.6  6 7 10 7.2  11 8.4  12 7.3 /
\arr{12.3 0.45} {12.7 1.05} 
\arr{1.3 1.05} {1.7 0.45} 
\arr{2.3 0.45} {2.7 1.05} 
\arr{3.3 1.05} {3.7 0.45} 
\arr{4.3 0.45} {4.7 1.05} 
\arr{5.3 1.05} {5.7 0.45} 
\arr{6.3 0.45} {6.7 1.05} 
\arr{7.3 1.05} {7.7 0.45} 
\arr{8.3 0.45} {8.7 1.05} 
\arr{9.3 1.05} {9.7 0.45} 
\arr{10.3 0.45} {10.7 1.05} 
\arr{11.3 1.05} {11.7 0.45} 
\arr{12.3 2.55} {12.7 1.95} 
\arr{1.3 1.95} {1.6 2.4} 
\arr{2.3 2.55} {2.7 1.95} 
\arr{3.3 1.95} {3.6 2.4} 
\arr{4.3 2.55} {4.7 1.95} 
\arr{5.3 1.95} {5.7 2.55} 
\arr{6.4 2.4 } {6.7 1.95} 
\arr{7.3 1.95} {7.6 2.4 } 
\arr{8.4 2.4 } {8.7 1.95} 
\arr{9.3 1.95} {9.7 2.55} 
\arr{10.3 2.55} {10.7 1.95} 
\arr{11.3 1.95} {11.7 2.55} 
\arr{12.4 3.4} {12.7 3.7} 
\arr{1.3 3.7} {1.6 3.4} 
\arr{2.3 3.3} {2.6 3.6} 
\arr{3.3 3.7} {3.6 3.4} 
\arr{4.4 3.4} {4.7 3.7} 
\arr{5.3 3.7} {5.6 3.4} 
\arr{6.4 3.4} {6.6 3.6} 
\arr{7.4 3.6} {7.6 3.4} 
\arr{8.4 3.4} {8.7 3.7} 
\arr{9.3 3.7} {9.7 3.3} 
\arr{10.4 3.4} {10.7 3.7} 
\arr{11.4 3.6} {11.7 3.3} 
\arr{12.3 4.7} {12.7 4.3} 
\arr{1.3 4.3} {1.7 4.7} 
\arr{2.3 4.7} {2.6 4.4} 
\arr{3.3 4.3} {3.7 4.7} 
\arr{4.3 4.7} {4.7 4.3} 
\arr{5.3 4.3} {5.7 4.7} 
\arr{6.3 4.7} {6.6 4.4} 
\arr{7.4 4.4} {7.7 4.7} 
\arr{8.3 4.7} {8.7 4.3} 
\arr{9.3 4.3} {9.7 4.7} 
\arr{10.3 4.7} {10.7 4.3} 
\arr{11.4 4.4} {11.7 4.7} 
\arr{12.3 5.3} {12.7 5.7} 
\arr{1.3 5.7} {1.7 5.3} 
\arr{2.3 5.3} {2.7 5.7} 
\arr{3.3 5.7} {3.7 5.3} 
\arr{4.3 5.3} {4.7 5.7} 
\arr{5.3 5.7} {5.7 5.3} 
\arr{6.3 5.3} {6.7 5.7} 
\arr{7.3 5.7} {7.7 5.3} 
\arr{8.3 5.3} {8.7 5.7} 
\arr{9.3 5.7} {9.7 5.3} 
\arr{10.3 5.3} {10.7 5.7} 
\arr{11.3 5.7} {11.7 5.3} 
\arr{12.3 6.7} {12.7 6.3} 
\arr{1.3 6.3} {1.7 6.7} 
\arr{2.3 6.7} {2.7 6.3} 
\arr{3.3 6.3} {3.7 6.7} 
\arr{4.3 6.7} {4.7 6.3} 
\arr{5.3 6.3} {5.7 6.7} 
\arr{6.3 6.7} {6.7 6.3} 
\arr{7.3 6.3} {7.7 6.7} 
\arr{8.3 6.7} {8.7 6.3} 
\arr{9.3 6.3} {9.7 6.7} 
\arr{10.3 6.7} {10.7 6.3} 
\arr{11.3 6.3} {11.7 6.7} 

\arr{2.3 7.3} {2.7 7.7} 
\arr{3.3 7.7} {3.7 7.3} 
\arr{10.3 7.3} {10.7 7.7} 
\arr{11.3 7.7} {11.7 7.3} 

\arr{12.4 2.92}{12.7 2.86}
\arr{1.3 2.86}{1.6 2.92}
\arr{2.4 2.92}{2.7 2.86}
\arr{3.3 2.86}{3.6 2.92}
\arr{4.4 2.92}{4.7 2.86}
\arr{5.3 2.86}{5.6 2.92}
\arr{6.4 2.92}{6.7 2.86}
\arr{7.3 2.86}{7.6 2.92}
\arr{8.4 2.92}{8.7 2.86}
\arr{9.3 2.86}{9.6 2.92}
\arr{10.4 2.92}{10.7 2.86}
\arr{11.3 2.86}{11.6 2.92}
\setdots<2pt>
\plot 1 0  1.3 0 /
\plot 2.7 0  3.3 0 /
\plot 4.7 0  5.3 0 /
\plot 6.7 0  7.3 0 /
\plot 8.7 0  9.3 0 /
\plot 10.7 0  11.3 0 /
\plot 12.7 0  13 0 /
\plot 1 7  1.3 7 /
\plot 4.7 7  5.3 7 /
\plot 6.7 7  7.3 7 /
\plot 8.7 7  9.3 7 /
\plot 12.7 7  13 7 /
\setdashes <1.5mm>
\plot 1 -.5  1 1 /
\plot 1 4.75  1 5.3 /
\plot 1 6.75  1 8.5 /
\plot 13 -.5  13 1 /
\plot 13 4.75  13 5.3 /
\plot 13 6.75  13 8.5 /
\setshadegrid span <1.5mm>
\vshade   1   0 7 <,z,,> 
          2   0 7  <z,z,,> 
          3   0 8  <z,z,,>
          4   0 7  <z,z,,> 
          10  0 7  <z,z,,> 
          11  0 8  <z,z,,> 
          12   0 7 <z,z,,>
          13  0 7 /
\endpicture} 
$$
\end{figure}

\begin{ex}
  Consider the category $\mathcal S(5)$ of embeddings $(A\subset B)$
  of a submodule in a finite length module $B$ over $\Lambda=k[T]/(T^5)$.
  Among the categories $\mathcal S(n)$, this is the largest category
  of finite representation type.  We picture in Figure~\ref{fig-arq}
  the Auslander-Reiten
  quiver for $\mathcal S(5)$ from \cite[(6.5)]{rs}
  which contains the 50 indecomposable objects, up to isomorphy.

\medskip
For each of the objects in the Auslander-Reiten quiver,
we compute the socle tablau; it is pictured at the position of the
given object in Figure~\ref{fig-socle}.

\medskip
Let us focus on the case where $\ell=2$ and $r=2$, so we are interested
in those embeddings $M$ which have a 2 in the second row of their
socle tableau.  Consider the short exact sequence
$$0\longrightarrow P_2^3\stackrel f\longrightarrow
P_2^4\oplus P_1^2\longrightarrow P_1^3\longrightarrow 0$$
from Corollary~\ref{cor-cok-hom}.  We have labelled the objects
$A=P_2^3$, $B=P_2^4$, $B'=P_1^2$ and $C=P_1^3$ in the diagram.
The region in which $M$ can possible occur is encircled
in Figure~\ref{fig-socle}.

\end{ex}


\begin{figure}[ht]
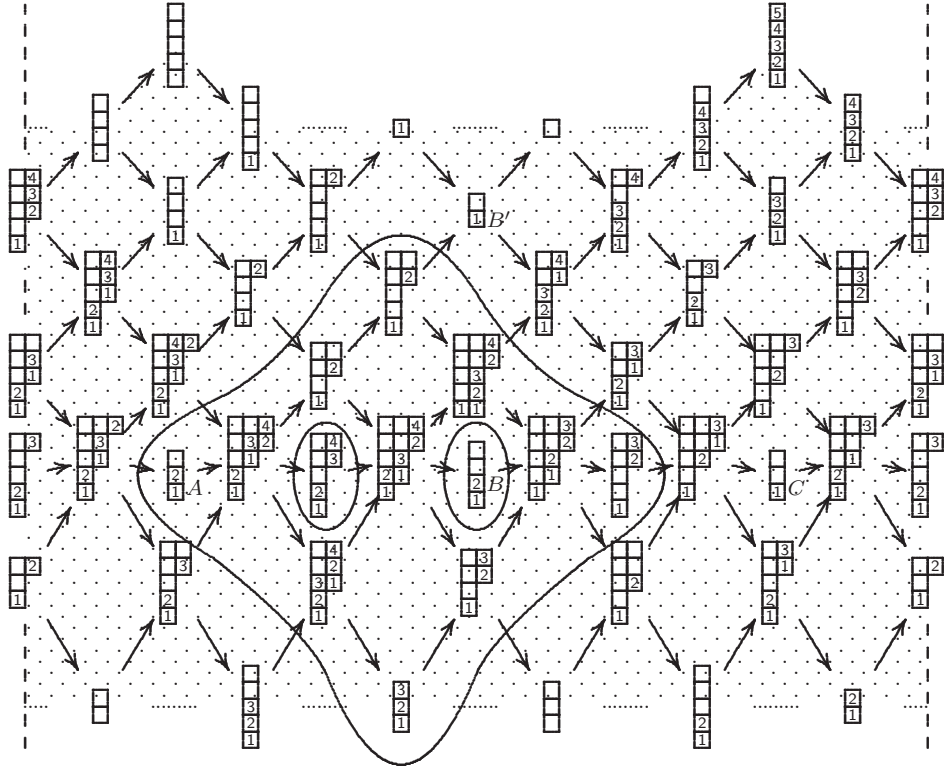

  \caption{The socle tableaux for the objects in $\mathcal S(5)$}
  \label{fig-socle}
$$\hbox{\beginpicture
\setcoordinatesystem units <1cm,1.1cm>
\multiput{\ssq} at 11.9 -.1  11.9 .1 /
\put{\num2} at 12 .1
\put{\num1} at 12 -.1 
\multiput{\ssq} at 1.9 -.1  1.9 .1 /  
\multiput{\ssq} at 3.9 -.4  3.9 -.2  3.9 0  3.9 .2  3.9 .4 /
\put{\num3} at 4 0
\put{\num2} at 4 -.2
\put{\num1} at 4 -.4
\multiput{\ssq} at 5.9 -.2  5.9 0  5.9 .2 /
\put{\num3} at 6 .2
\put{\num2} at 6 0
\put{\num1} at 6 -.2
\multiput{\ssq} at 7.9 -.2  7.9 0  7.9 .2 /
\multiput{\ssq} at 9.9 -.4  9.9 -.2  9.9 0  9.9 .2  9.9 .4 /
\put{\num2} at 10 -.2
\put{\num1} at 10 -.4
%
%
\multiput{\ssq} at .8 1.3  .8 1.5  .8 1.7  1 1.7 /
\put{\num2} at 1.1 1.7
\put{\num1} at .9 1.3
\multiput{\ssq} at  2.8 1.1  2.8 1.3  2.8 1.5  2.8 1.7  2.8 1.9  3 1.7  3 1.9 /
\put{\num3} at 3.1 1.7
\put{\num2} at 2.9 1.3
\put{\num1} at 2.9 1.1
\multiput{\ssq} at 4.8 1.1  4.8 1.3  4.8 1.5  4.8 1.7  4.8 1.9  
                5 1.9  5 1.5  5 1.7 /
\multiput{\num1} at 4.9 1.1  5.1 1.5 /
\multiput{\num2} at 4.9 1.3  5.1 1.7 /
\put{\num4} at 5.1 1.9  
\put{\num3} at 4.9 1.5  
\multiput{\ssq} at 6.8 1.2  6.8 1.4  6.8 1.6  6.8 1.8  7 1.6  7 1.8 /
\put{\num3} at 7.1 1.8
\put{\num2} at 7.1 1.6
\put{\num1} at 6.9 1.2
\multiput{\ssq} at  9 1.5  9 1.7  9 1.9  
                8.8 1.1  8.8 1.3  8.8 1.5  8.8 1.7  8.8 1.9 /
\put{\num2} at 9.1 1.5
\put{\num1} at 8.9 1.1 
\multiput{\ssq} at 10.8 1.1  10.8 1.3  10.8 1.5  10.8 1.7  10.8 1.9  
                                11 1.7  11 1.9 /
\multiput{\num1} at 11.1 1.7  10.9 1.1 /
\put{\num3} at 11.1 1.9 
\put{\num2} at 10.9 1.3 
\multiput{\ssq} at 12.8 1.3  12.8 1.5  12.8 1.7  13 1.7 /
\put{\num2} at 13.1 1.7
\put{\num1} at 12.9 1.3
%
%
%
\multiput{\ssq} at .8 2.4  .8 2.6  .8 2.8  .8 3  .8 3.2  1 3.2 /
\put{\num3} at 1.1 3.2
\put{\num2} at .9 2.6
\put{\num1} at .9 2.4
\multiput{\ssq} at 1.7 2.6  1.7 2.8  1.7 3  1.7 3.2  1.7 3.4  
                1.9 3.4 1.9 3  1.9 3.2  2.1 3.4 /
\multiput{\num1} at 1.8 2.6  2 3 /  
\multiput{\num2} at 1.8 2.8  2.2 3.4 /
\put{\num3} at 2 3.2 
\multiput{\ssq} at 2.9 2.6  2.9 2.8  2.9 3 /
\put{\num2} at 3 2.8
\put{\num1} at 3 2.6
\multiput{\ssq} at 3.7 2.6  3.7 2.8  3.7 3  3.7 3.2  3.7 3.4  
                3.9 3.4  3.9 3  3.9 3.2   4.1 3.4  4.1 3.2 /
\multiput{\num1} at 3.8 2.6  4 3 /
\multiput{\num2} at 4.2 3.2  3.8 2.8 /
\put{\num3} at 4 3.2
\put{\num4} at 4.2 3.4 
\multiput{\ssq} at 4.8 2.4  4.8 2.6  4.8 2.8  4.8 3  4.8 3.2  5 3.2  5 3 /
\put{\num4} at 5.1 3.2
\put{\num3} at 5.1 3
\put{\num2} at 4.9 2.6
\put{\num1} at 4.9 2.4
\multiput{\ssq} at 5.7 2.6  5.7 2.8  5.7 3  5.7 3.2  5.7 3.4  
                5.9 2.8  5.9 3  5.9 3.2  5.9 3.4
                6.1 3.4  6.1 3.2 /
\multiput{\num1} at 5.8 2.6  6 2.8 / 
\multiput{\num2} at 6.2 3.2  5.8 2.8 /
\put{\num3} at 6 3
\put{\num4} at 6.2 3.4 
\multiput{\ssq} at 6.9 2.5  6.9 2.7  6.9 2.9  6.9 3.1 /
\put{\num2} at 7 2.7
\put{\num1} at 7 2.5
\multiput{\ssq} at 7.7 2.6  7.7 2.8  7.7 3  7.7 3.2  7.7 3.4  
                7.9 2.8  7.9 3  7.9 3.2  7.9 3.4   8.1 3.2  8.1 3.4 /
\multiput{\num1} at 7.8 2.6  8 2.8 /
\multiput{\num2} at 8.2 3.2  8 3 /
\put{\num3} at 8.2 3.4  
\multiput{\ssq} at 8.8 2.4  8.8 2.6  8.8 2.8  8.8 3  8.8 3.2  9 3  9 3.2 /
\put{\num3} at 9.1 3.2 
\put{\num2} at 9.1 3
\put{\num1} at 8.9 2.4
\multiput{\ssq} at 9.7 2.6  9.7 2.8  9.7 3  9.7 3.2  9.7 3.4 
                9.9 3  9.9 3.2  9.9 3.4  10.1 3.2  10.1 3.4 /
\multiput{\num1} at 10.2 3.2  9.8 2.6 /
\put{\num2} at 10 3
\put{\num3} at 10.2 3.4 
\multiput{\ssq} at 10.9 2.6  10.9 2.8  10.9 3 /
\put{\num1} at 11 2.6
\multiput{\ssq} at 11.7 2.6  11.7 2.8  11.7 3  11.7 3.2  11.7 3.4  
                11.9 3.4  11.9 3  11.9 3.2       12.1 3.4 /
\multiput{\num1} at 11.8 2.6  12 3 / 
\put{\num2} at 11.8 2.8 
\put{\num3} at 12.2 3.4
\multiput{\ssq} at 12.8 2.4  12.8 2.6  12.8 2.8  12.8 3  12.8 3.2  13 3.2 /
\put{\num3} at 13.1 3.2
\put{\num2} at 12.9 2.6
\put{\num1} at 12.9 2.4
\multiput{\ssq} at .8 3.6  .8 3.8  .8 4  .8 4.2  .8 4.4  1 4  1 4.4  1 4.2 /
\multiput{\num1} at .9 3.6  1.1 4 /
\put{\num2} at .9 3.8
\put{\num3} at 1.1 4.2 
\multiput{\ssq} at  2.7 3.6  2.7 3.8  2.7 4  2.7 4.2  2.7 4.4  
                2.9 4.4  2.9 4  2.9 4.2      3.1 4.4 /
\multiput{\num1} at 3 4  2.8 3.6 /
\multiput{\num2} at 3.2 4.4  2.8 3.8 /
\put{\num3} at 3 4.2
\put{\num4} at 3 4.4 
\multiput{\ssq} at 4.8 3.7  4.8 3.9  4.8 4.1  4.8 4.3  5 4.3  5 4.1 /
\put{\num2} at 5.1 4.1
\put{\num1} at 4.9 3.7
\multiput{\ssq} at  6.7 3.6  6.7 3.8  6.7 4  6.7 4.2  6.7 4.4  
                6.9 3.6  6.9 3.8  6.9 4  6.9 4.2  6.9 4.4   7.1 4.2  7.1 4.4 /
\multiput{\num1} at 6.8 3.6  7 3.6 /
\multiput{\num2} at 7 3.8  7.2 4.2 /
\put{\num3} at 7 4
\put{\num4} at 7.2 4.4
\multiput{\ssq} at   9 4.3  9 4.1  8.8 3.7  8.8 3.9  8.8 4.1  8.8 4.3 /
\multiput{\num1} at 8.9 3.7  9.1 4.1 /
\put{\num2} at 8.9 3.9 
\put{\num3} at 9.1 4.3 
\multiput{\ssq} at    10.7 3.6  10.7 3.8  10.7 4  10.7 4.2  10.7 4.4  
                10.9 4.4   10.9 4  10.9 4.2   11.1 4.4 /
\put{\num3} at 11.2 4.4 
\put{\num2} at 11 4
\put{\num1} at 10.8 3.6
\multiput{\ssq} at 12.8 3.6  12.8 3.8  12.8 4  12.8 4.2  12.8 4.4  13 4  13 4.4  13 4.2 /
\multiput{\num1} at 12.9 3.6  13.1 4 /
\put{\num2} at 12.9 3.8
\put{\num3} at 13.1 4.2 
%
%
%
\multiput{\ssq} at 1.8 4.6  1.8 4.8  1.8 5  1.8 5.2  1.8 5.4  
                2 5.4  2 5  2 5.2 /
\multiput{\num1} at 2.1 5  1.9 4.6 /
\put{\num2} at 1.9 4.8
\put{\num3} at 2.1 5.2
\put{\num4} at 2.1 5.4 
\multiput{\ssq} at 3.8 4.7  3.8 4.9  3.8 5.1  3.8 5.3  4 5.3 /
\put{\num2} at 4.1 5.3
\put{\num1} at 3.9 4.7
\multiput{\ssq} at 5.8 4.6  5.8 4.8  5.8 5  5.8 5.2  5.8 5.4  6 5.2  6 5.4 /
\put{\num2} at 6.1 5.2 
\put{\num1} at 5.9 4.6
\multiput{\ssq} at  7.8 4.6  7.8 4.8  7.8 5  7.8 5.2  7.8 5.4  8 5.2  8 5.4 /
\multiput{\num1} at 7.9 4.6  8.1 5.2 /
\put{\num2} at 7.9 4.8
\put{\num3} at 7.9 5
\put{\num4} at 8.1 5.4
\multiput{\ssq} at  9.8 4.7  9.8 4.9  9.8 5.1  9.8 5.3  10 5.3 /
\put{\num3} at 10.1 5.3
\put{\num2} at 9.9 4.9
\put{\num1} at 9.9 4.7
\multiput{\ssq} at 11.8 4.6  11.8 4.8  11.8 5  11.8 5.2  11.8 5.4  
                12 5.4  12 5  12 5.2 /
\put{\num3} at 12.1 5.2
\put{\num2} at 12.1 5
\put{\num1} at 11.9 4.6
\multiput{\ssq} at .8 5.6  .8 5.8  .8 6  .8 6.2  .8 6.4  1 6.4  1 6  1 6.2 /
\put{\num4} at 1.1 6.4
\put{\num3} at 1.1 6.2
\put{\num2} at 1.1 6
\put{\num1} at .9 5.6
\multiput{\ssq} at  2.9 5.7  2.9 5.9  2.9 6.1  2.9 6.3 /
\put{\num1} at 3 5.7
\multiput{\ssq} at  4.8 5.6  4.8 5.8  4.8 6  4.8 6.2  4.8 6.4  5 6.4 /
\put{\num2} at 5.1 6.4
\put{\num1} at 4.9 5.6
\multiput{\ssq} at  6.9 5.9  6.9 6.1 /
\put{\num1} at 7 5.9
\multiput{\ssq} at  8.8 5.6  8.8 5.8  8.8 6  8.8 6.2  8.8 6.4  9 6.4 /
\put{\num4} at 9.1 6.4
\put{\num3} at 8.9 6
\put{\num2} at 8.9 5.8
\put{\num1} at 8.9 5.6
\multiput{\ssq} at 10.9 5.7  10.9 5.9  10.9 6.1  10.9 6.3 /
\put{\num3} at 11 6.1
\put{\num2} at 11 5.9
\put{\num1} at 11 5.7
\multiput{\ssq} at 12.8 5.6  12.8 5.8  12.8 6  12.8 6.2  12.8 6.4  13 6.4  13 6  13 6.2 /
\put{\num4} at 13.1 6.4
\put{\num3} at 13.1 6.2
\put{\num2} at 13.1 6
\put{\num1} at 12.9 5.6
%
%
%
\multiput{\ssq} at   1.9 6.7  1.9 6.9  1.9 7.1  1.9 7.3 /
\multiput{\ssq} at    2.9 7.6  2.9 7.8  2.9 8  2.9 8.2  2.9 8.4 /
\multiput{\ssq} at  3.9 6.6  3.9 6.8  3.9 7  3.9 7.2  3.9 7.4 /
\put{\num1} at 4 6.6
\multiput{\ssq} at  5.9 7 /
\put{\num1} at 6 7
\multiput{\ssq} at   7.9 7 /
\multiput{\ssq} at  9.9 6.6  9.9 6.8  9.9 7  9.9 7.2  9.9 7.4 /
\put{\num4} at 10 7.2 
\put{\num3} at 10 7
\put{\num2} at 10 6.8
\put{\num1} at 10 6.6
\multiput{\ssq} at  10.9 7.6  10.9 7.8  10.9 8  10.9 8.2  10.9 8.4 /
\put{\num5} at 11 8.4
\put{\num4} at 11 8.2
\put{\num3} at 11 8
\put{\num2} at 11 7.8
\put{\num1} at 11 7.6
\multiput{\ssq} at  11.9 6.7  11.9 6.9  11.9 7.1  11.9 7.3 /
\put{\num4} at 12 7.3
\put{\num3} at 12 7.1
\put{\num2} at 12 6.9
\put{\num1} at 12 6.7
\arr{12.3 0.45} {12.7 1.05} 
\arr{1.3 1.05} {1.7 0.45} 
\arr{2.3 0.45} {2.7 1.05} 
\arr{3.3 1.05} {3.7 0.45} 
\arr{4.3 0.45} {4.7 1.05} 
\arr{5.3 1.05} {5.7 0.45} 
\arr{6.3 0.45} {6.7 1.05} 
\arr{7.3 1.05} {7.7 0.45} 
\arr{8.3 0.45} {8.7 1.05} 
\arr{9.3 1.05} {9.7 0.45} 
\arr{10.3 0.45} {10.7 1.05} 
\arr{11.3 1.05} {11.7 0.45} 
\arr{12.3 2.55} {12.7 1.95} 
\arr{1.3 1.95} {1.6 2.4} 
\arr{2.3 2.55} {2.7 1.95} 
\arr{3.3 1.95} {3.6 2.4} 
\arr{4.3 2.55} {4.7 1.95} 
\arr{5.3 1.95} {5.7 2.55} 
\arr{6.4 2.4 } {6.7 1.95} 
\arr{7.3 1.95} {7.6 2.4 } 
\arr{8.4 2.4 } {8.7 1.95} 
\arr{9.3 1.95} {9.7 2.55} 
\arr{10.3 2.55} {10.7 1.95} 
\arr{11.3 1.95} {11.7 2.55} 
\arr{12.4 3.4} {12.7 3.7} 
\arr{1.3 3.7} {1.6 3.4} 
\arr{2.3 3.3} {2.6 3.6} 
\arr{3.3 3.7} {3.6 3.4} 
\arr{4.4 3.4} {4.7 3.7} 
\arr{5.3 3.7} {5.6 3.4} 
\arr{6.4 3.4} {6.6 3.6} 
\arr{7.4 3.6} {7.6 3.4} 
\arr{8.4 3.4} {8.7 3.7} 
\arr{9.3 3.7} {9.7 3.3} 
\arr{10.4 3.4} {10.7 3.7} 
\arr{11.4 3.6} {11.7 3.3} 
\arr{12.3 4.7} {12.7 4.3} 
\arr{1.3 4.3} {1.7 4.7} 
\arr{2.3 4.7} {2.6 4.4} 
\arr{3.3 4.3} {3.7 4.7} 
\arr{4.3 4.7} {4.7 4.3} 
\arr{5.3 4.3} {5.7 4.7} 
\arr{6.3 4.7} {6.6 4.4} 
\arr{7.4 4.4} {7.7 4.7} 
\arr{8.3 4.7} {8.7 4.3} 
\arr{9.3 4.3} {9.7 4.7} 
\arr{10.3 4.7} {10.7 4.3} 
\arr{11.4 4.4} {11.7 4.7} 
\arr{12.3 5.3} {12.7 5.7} 
\arr{1.3 5.7} {1.7 5.3} 
\arr{2.3 5.3} {2.7 5.7} 
\arr{3.3 5.7} {3.7 5.3} 
\arr{4.3 5.3} {4.7 5.7} 
\arr{5.3 5.7} {5.7 5.3} 
\arr{6.3 5.3} {6.7 5.7} 
\arr{7.3 5.7} {7.7 5.3} 
\arr{8.3 5.3} {8.7 5.7} 
\arr{9.3 5.7} {9.7 5.3} 
\arr{10.3 5.3} {10.7 5.7} 
\arr{11.3 5.7} {11.7 5.3} 
\arr{12.3 6.7} {12.7 6.3} 
\arr{1.3 6.3} {1.7 6.7} 
\arr{2.3 6.7} {2.7 6.3} 
\arr{3.3 6.3} {3.7 6.7} 
\arr{4.3 6.7} {4.7 6.3} 
\arr{5.3 6.3} {5.7 6.7} 
\arr{6.3 6.7} {6.7 6.3} 
\arr{7.3 6.3} {7.7 6.7} 
\arr{8.3 6.7} {8.7 6.3} 
\arr{9.3 6.3} {9.7 6.7} 
\arr{10.3 6.7} {10.7 6.3} 
\arr{11.3 6.3} {11.7 6.7} 

\arr{2.3 7.3} {2.7 7.7} 
\arr{3.3 7.7} {3.7 7.3} 
\arr{10.3 7.3} {10.7 7.7} 
\arr{11.3 7.7} {11.7 7.3} 

\arr{12.4 2.92}{12.7 2.86}
\arr{1.3 2.86}{1.6 2.92}
\arr{2.4 2.92}{2.7 2.86}
\arr{3.3 2.86}{3.6 2.92}
\arr{4.4 2.92}{4.7 2.86}
\arr{5.3 2.86}{5.6 2.92}
\arr{6.4 2.92}{6.7 2.86}
\arr{7.3 2.86}{7.6 2.92}
\arr{8.4 2.92}{8.7 2.86}
\arr{9.3 2.86}{9.6 2.92}
\arr{10.4 2.92}{10.7 2.86}
\arr{11.3 2.86}{11.6 2.92}
\setdots<2pt>
\plot 1 0  1.3 0 /
\plot 2.7 0  3.3 0 /
\plot 4.7 0  5.3 0 /
\plot 6.7 0  7.3 0 /
\plot 8.7 0  9.3 0 /
\plot 10.7 0  11.3 0 /
\plot 12.7 0  13 0 /
\plot 1 7  1.3 7 /
\plot 4.7 7  5.3 7 /
\plot 6.7 7  7.3 7 /
\plot 8.7 7  9.3 7 /
\plot 12.7 7  13 7 /
\setdashes <1.5mm>
\plot 1 -.5  1 1 /
\plot 1 4.75  1 5.3 /
\plot 1 6.75  1 8.5 /
\plot 13 -.5  13 1 /
\plot 13 4.75  13 5.3 /
\plot 13 6.75  13 8.5 /
\setshadegrid span <1.5mm>
\vshade   1   0 7 <,z,,> 
          2   0 7  <z,z,,> 
          3   0 8  <z,z,,>
          4   0 7  <z,z,,> 
          10  0 7  <z,z,,> 
          11  0 8  <z,z,,> 
          12   0 7 <z,z,,>
          13  0 7 /
\setsolid
\setquadratic
\plot 3.5 1.8  2.5 2.8  3.5 3.7  4.3 4.2  5 5  6 5.7  7 5  7.7 4.2  8.5 3.7  9.5 2.8  
      8.5 1.8  7.5 1  7 .4  6 -.7  5 .4  4.5 1  3.5 1.8 /
\ellipticalarc axes ratio 3:5 360 degrees from 7 3.44 center at 7 2.79
\ellipticalarc axes ratio 3:5 360 degrees from 5 3.44 center at 5 2.79
\put{$\scriptstyle C$} at 11.25 2.65
\put{$\scriptstyle B'$} at 7.3 5.9
\put{$\scriptstyle B$} at 7.25 2.7
\put{$\scriptstyle A$} at 3.25 2.65
\endpicture} 
$$
\end{figure}

\bigskip
We note that in general (and in this example),
the entries in the socle tableau
determine different regions in the Auslander-Reiten quiver
than the entries in the  LR-tableau or in the dual LR-tableau, see
\cite[Section~6]{s-lrtableau}.

\subsection{Positioning objects in the representation space}


\begin{figure}[ht]
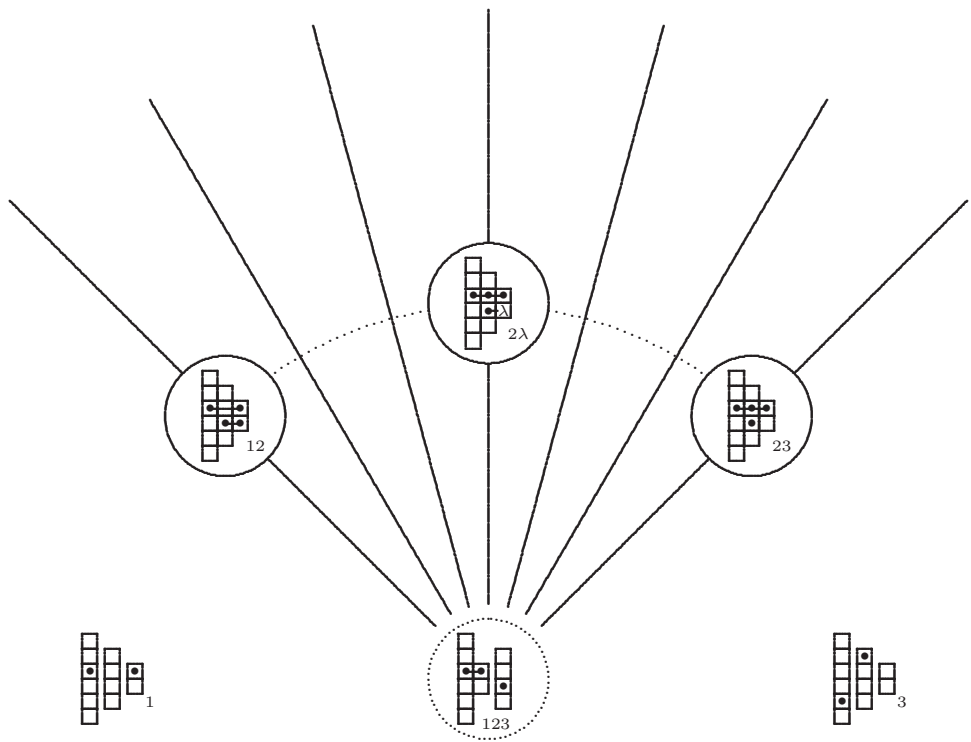

  \caption{The representation space $\mathbb V_{42,42}^{642}$}
  \label{fig-rep-space}
$$
\hbox{\beginpicture
\setcoordinatesystem units <1cm,1cm>
\multiput{} at -.5 -1  10.5 9 /
\multiput{\ssq} at -.4 -.4  -.4 -.2  -.4 0  -.4 .2  -.4 .4  -.4 .6
                   -.1 -.2  -.1 0  -.1 .2  -.1 .4
                   .2 0  .2 .2 /
\multiput{\bul} at -.3 .2  .3 .2 /
\put{$\ssize 1$} at .5 -.2
\multiput{\ssq} at 4.6 -.4  4.6 -.2  4.6 0  4.6 .2  4.6 .4  4.6 .6
                   4.8 0  4.8 .2
                   5.1 -.2  5.1 0  5.1 .2  5.1 .4 /
\multiput{\bul} at 4.7 .2  4.9 .2  5.2 0 /
\plot 4.7 .2  4.9 .2 /
\put{$\ssize123$} at 5.1 -.5
\multiput{\ssq} at 9.6 -.4  9.6 -.2  9.6 0  9.6 .2  9.6 .4  9.6 .6
                   9.9 -.2  9.9 0  9.9 .2  9.9  .4
                   10.2 0  10.2 .2 /
\multiput{\bul} at 9.7 -.2  10 .4 /
\put{$\ssize 3$} at 10.5 -.2
\multiput{\ssq} at 4.7 4.6  4.7 4.8  4.7 5  4.7 5.2  4.7 5.4  4.7 5.6
                   4.9 4.8  4.9 5  4.9 5.2  4.9 5.4
                   5.1 5  5.1 5.2 /
\multiput{\bul} at 4.8 5.2  5 5.2  5.2 5.2  5 5 /
\put{$\scriptscriptstyle\lambda$} at 5.2 5
\put{$\ssize 2\lambda$} at 5.4 4.7
\plot 4.8 5.2  5.2 5.2 /
\plot 5 5  5.12 5 /
\multiput{\ssq} at 1.2 3.1  1.2 3.3  1.2 3.5  1.2 3.7  1.2 3.9  1.2 4.1
                   1.4 3.3  1.4 3.5  1.4 3.7  1.4 3.9
                   1.6 3.5  1.6 3.7 /
\multiput{\bul} at 1.3 3.7  1.7 3.7  1.5 3.5  1.7 3.5 /
\plot 1.3 3.7  1.7 3.7 /
\plot 1.5 3.5  1.7 3.5 /
\put{$\ssize12$} at 1.9 3.2
\multiput{\ssq} at 8.2 3.1  8.2 3.3  8.2 3.5  8.2 3.7  8.2 3.9  8.2 4.1
                   8.4 3.3  8.4 3.5  8.4 3.7  8.4 3.9
                   8.6 3.5  8.6 3.7 /
\multiput{\bul} at 8.3 3.7  8.5 3.7  8.7 3.7  8.5 3.5 /
\plot 8.3 3.7  8.7 3.7 /
\put{$\ssize23$} at 8.9 3.2
\setdots<2pt>
\circulararc 360 degrees from 5 .9  center at  5 .1
\setsolid
\circulararc 360 degrees from 1.5 4.4  center at  1.5 3.6
\circulararc 360 degrees from 5 5.9  center at  5 5.1
\circulararc 360 degrees from 8.5 4.4  center at  8.5 3.6
\setdots<3pt>
\circulararc 26 degrees from 4.2 5  center at  5 .1
\circulararc -26 degrees from 5.8 5  center at  5 .1
\setsolid
\plot 5 1.1  5 4.3 /
\plot 5 5.9  5 9 /
\plot 5.26 1.066  7.33 8.79 /
\plot 5.5 .966  9.5 7.804 /
\plot 5.707 .807  7.93 3.03 /
\plot 9.066 4.166  11.364 6.464 /
\plot 4.74 1.066  2.67 8.79 /
\plot 4.5 .966  .5  7.804 /
\plot 4.293 .807  2.07 3.03 /
\plot .934 4.166  -1.364 6.464 /
\endpicture}$$
\end{figure}

Among the categories of type $\mathcal S(n)$, the first category of
infinite representation type is $\mathcal S(6)$ (see \cite{rs}).
And in $\mathcal S(6)$,
the smallest parametrized family of pairwise non-isomorphic
indecomposable embeddings $(A\subset B)$ occurs when $A\cong N_{42}$
and $B\cong N_{642}$; in this case $B/A\cong N_{42}$.

\medskip
We use the notation from \cite[Example~4.5]{ks-poles}
where we have studied the representation space
$V=\mathbb V_{42,42}^{642}$, it is the constructible variety of all
$f\in\Hom_k(N_{42},N_{642})$ which give rise to a short exact sequence
of $\Lambda$-modules $0\to N_{42}\stackrel f\to N_{642}\to N_{42}\to 0$
where $\Lambda=k[T]/(T^6)$.
This variety $V$ has three irreducible components, they are given by
the three LR-tableaux of shape $(42,642,42)$, but also by the three
socle tableaux of the same shape:
\newcommand\SixOne{\begin{picture}(9,18)
    \multiput(0,0)(0,3)6{\smbox}
    \multiput(3,6)(0,3)4{\smbox}
    \multiput(6,12)(0,3)2{\smbox}
    \put(6,15){\numbox2}
    \put(6,12){\numbox1}
    \put(3,9){\numbox4}
    \put(3,6){\numbox3}
    \put(0,3){\numbox2}
    \put(0,0){\numbox1}
  \end{picture}
}
\newcommand\SixTwo{\begin{picture}(9,18)
    \multiput(0,0)(0,3)6{\smbox}
    \multiput(3,6)(0,3)4{\smbox}
    \multiput(6,12)(0,3)2{\smbox}
    \put(6,15){\numbox4}
    \put(6,12){\numbox2}
    \put(3,9){\numbox3}
    \put(3,6){\numbox1}
    \put(0,3){\numbox2}
    \put(0,0){\numbox1}
  \end{picture}
}
\newcommand\SixThree{\begin{picture}(9,18)
    \multiput(0,0)(0,3)6{\smbox}
    \multiput(3,6)(0,3)4{\smbox}
    \multiput(6,12)(0,3)2{\smbox}
    \put(6,15){\numbox4}
    \put(6,12){\numbox3}
    \put(3,9){\numbox2}
    \put(3,6){\numbox1}
    \put(0,3){\numbox2}
    \put(0,0){\numbox1}
  \end{picture}
}

$$ \Sigma_1:\;\raisebox{-7mm}{\SixOne},\qquad
\Sigma_2:\;\raisebox{-7mm}{\SixTwo}, \qquad
\Sigma_3:\;\raisebox{-7mm}{\SixThree}$$

\medskip
Denote by $M_x$ the embedding in Figure~\ref{fig-rep-space} with subscript $x$.
The objects $M_3$, $M_{23}$ and $M_{123}$ have socle tableau $\Sigma_3$.
In the representation space, the embeddings isomorphic to
$M_3$ ($M_{23}$, $M_{123}$) form orbits of dimension 4 (3, 2), respectively.
Together, the three orbits form a closed subset, hence an irreducible
component, of $V$.

\medskip
The objects $M_{12}$ and $M_{2\lambda}$ have socle tableau $\Sigma_2$;
each forms an orbit of dimension 3, although there is a one-parameter
family of such orbits.  The closure of the union of those orbits
also contains $M_{23}$ and $M_{123}$.  The remaining embedding $M_1$
has socle tableau $\Sigma_1$, its orbit has dimension 4 and contains
in its closure $M_{12}$ and $M_{123}$.

\medskip
This partition of the representation space $V$ into regions given by
socle tableaux is dual to the partition given by LR-tableaux
(where an object $M_x$ with subscript $x=i*$ belongs to the region
given by the LR-tableau $\Gamma_i$).  This duality is expected as
socle tableaux are determined by the LR-tableaux of the dual embeddings
(Theorem~\ref{thm-lr-soc-hom}).

\subsection{Socle tableaux and tableau switching.}

We have seen in Section~\ref{sec-soc-lr-hom} that there is a one-to-one
correspondence between the sets of socle tableaux of shape $(\alpha,\beta,\gamma)$
and the set of LR-tableaux of shape $(\gamma,\beta,\alpha)$, given by
mapping the socle tableau of an embedding to the LR-tableau of the dual
embedding.

\smallskip
 Thanks to an~anonymous referee we observed that on the combinatorial level,
this correspondence in our examples is given by \boldit{tableau switching,}
see \cite{bss} (or \cite{ty} for an application to Schubert calculus).
We illustrate the tableau switching algorithm in the example from
Section~\ref{example-two}.

\smallskip
Let $S$ be the tableau of shape $\gamma=(31)$ where each box in row $i$
has entry $i$.  To be consistent with the notation in \cite{bss},
let $T$ be the tableau of shape $\beta\setminus\gamma$ where
the entry of a box is $5-i$ if $i$ is in the corresponding box in the socle
tableau $\Sigma_2$ (from Section~\ref{example-two}).

\newcommand\bfnumbox[1]{{\thicklines
  \put(0,0)\smbox%
  \put(0,0){\makebox(\value{boxsize},\value{boxsize})[c]{%
      $\smallentryformat{\bf#1}$}}}}

$$
\Sigma_2:\;
\raisebox{-5mm}\STtwo
\qquad
S:\;
\raisebox{-5mm}{\begin{picture}(9,15)
  \multiput(0,9)(3,0)2{\bfnumbox1}
  \put(0,6){\bfnumbox2}
  \put(0,3){\bfnumbox3}
\end{picture}}
\qquad 
T:\;\raisebox{-5mm}{\begin{picture}(9,15)
    \put(6,12){\numbox1}
    \put(3,9){\numbox2}
    \put(6,9){\numbox3}
    \put(3,6){\numbox4}
    \put(0,3){\numbox3}
    \put(0,0){\numbox4}
\end{picture}}
$$

In $S\cup T$, repeat switching neighboring entries from $S$  (in {\bf bold})
and $T$ such that
the entries from $T$ move up or to the left, and such that after each step,
the entries in $T$ and the entries in $S$ are weakly increasing in each row
and strictly increasing in each column.

$$S\cup T:\;
\raisebox{-5mm}{\begin{picture}(9,15)
    \put(0,12){\bfnumbox1}
    \put(3,12){\bfnumbox1}
    \put(6,12){\numbox1}
    \put(0,9){\bfnumbox2}
    \put(3,9){\numbox2}
    \put(6,9){\numbox3}
    \put(0,6){\bfnumbox3}
    \put(3,6){\numbox4}
    \put(0,3){\numbox3}
    \put(0,0){\numbox4}
\end{picture}}
\;\stackrel{3-3}\longmapsto\;
\raisebox{-5mm}{\begin{picture}(9,15)
    \put(0,12){\bfnumbox1}
    \put(3,12){\bfnumbox1}
    \put(6,12){\numbox1}
    \put(0,9){\bfnumbox2}
    \put(3,9){\numbox2}
    \put(6,9){\numbox3}
    \put(0,6){\numbox3}
    \put(3,6){\numbox4}
    \put(0,3){\bfnumbox3}
    \put(0,0){\numbox4}
\end{picture}}
\;\stackrel{3-4}\longmapsto\;
\raisebox{-5mm}{\begin{picture}(9,15)
    \put(0,12){\bfnumbox1}
    \put(3,12){\bfnumbox1}
    \put(6,12){\numbox1}
    \put(0,9){\bfnumbox2}
    \put(3,9){\numbox2}
    \put(6,9){\numbox3}
    \put(0,6){\numbox3}
    \put(3,6){\numbox4}
    \put(0,3){\numbox4}
    \put(0,0){\bfnumbox3}
\end{picture}}
\;\stackrel{2-2}\longmapsto\;
\raisebox{-5mm}{\begin{picture}(9,15)
    \put(0,12){\bfnumbox1}
    \put(3,12){\bfnumbox1}
    \put(6,12){\numbox1}
    \put(0,9){\numbox2}
    \put(3,9){\bfnumbox2}
    \put(6,9){\numbox3}
    \put(0,6){\numbox3}
    \put(3,6){\numbox4}
    \put(0,3){\numbox4}
    \put(0,0){\bfnumbox3}
\end{picture}}
\;\stackrel{1-1}\longmapsto\;
\raisebox{-5mm}{\begin{picture}(9,15)
    \put(0,12){\bfnumbox1}
    \put(3,12){\numbox1}
    \put(6,12){\bfnumbox1}
    \put(0,9){\numbox2}
    \put(3,9){\bfnumbox2}
    \put(6,9){\numbox3}
    \put(0,6){\numbox3}
    \put(3,6){\numbox4}
    \put(0,3){\numbox4}
    \put(0,0){\bfnumbox3}
\end{picture}}
\;\stackrel{1-1}\longmapsto
$$

$$
\stackrel{1-1}\longmapsto\;
\raisebox{-5mm}{\begin{picture}(9,15)
    \put(0,12){\numbox1}
    \put(3,12){\bfnumbox1}
    \put(6,12){\bfnumbox1}
    \put(0,9){\numbox2}
    \put(3,9){\bfnumbox2}
    \put(6,9){\numbox3}
    \put(0,6){\numbox3}
    \put(3,6){\numbox4}
    \put(0,3){\numbox4}
    \put(0,0){\bfnumbox3}
\end{picture}}
\stackrel{2-3}\longmapsto\;
\raisebox{-5mm}{\begin{picture}(9,15)
    \put(0,12){\numbox1}
    \put(3,12){\bfnumbox1}
    \put(6,12){\bfnumbox1}
    \put(0,9){\numbox2}
    \put(3,9){\numbox3}
    \put(6,9){\bfnumbox2}
    \put(0,6){\numbox3}
    \put(3,6){\numbox4}
    \put(0,3){\numbox4}
    \put(0,0){\bfnumbox3}
\end{picture}}
\stackrel{1-3}\longmapsto\;
\raisebox{-5mm}{\begin{picture}(9,15)
    \put(0,12){\numbox1}
    \put(3,12){\numbox3}
    \put(6,12){\bfnumbox1}
    \put(0,9){\numbox2}
    \put(3,9){\bfnumbox1}
    \put(6,9){\bfnumbox2}
    \put(0,6){\numbox3}
    \put(3,6){\numbox4}
    \put(0,3){\numbox4}
    \put(0,0){\bfnumbox3}
\end{picture}}
\stackrel{1-4}\longmapsto\;
\raisebox{-5mm}{\begin{picture}(9,15)
    \put(0,12){\numbox1}
    \put(3,12){\numbox3}
    \put(6,12){\bfnumbox1}
    \put(0,9){\numbox2}
    \put(3,9){\numbox4}
    \put(6,9){\bfnumbox2}
    \put(0,6){\numbox3}
    \put(3,6){\bfnumbox1}
    \put(0,3){\numbox4}
    \put(0,0){\bfnumbox3}
\end{picture}}
\;:{}^ST\cup S_T
$$

Thus, we do indeed obtain $\Gamma^*_2$ after 8 steps as the union of the empty tableau of shape
$\alpha=(42)$ with $S_T$.

$$
S_T: \;\raisebox{-5mm}{\begin{picture}(9,15)
    \put(6,12){\bfnumbox1}
    \put(6,9){\bfnumbox2}
    \put(3,6){\bfnumbox1}
    \put(0,0){\bfnumbox3}
\end{picture}}
\qquad \Gamma^*_2:\;\raisebox{-5mm}\LRtwoD
$$

\bigskip

  {\bf Conjecture:} The tableau switching algorithm always converts
  the socle tableau of an embedding into the LR-tableau of the dual
  embedding.

\medskip
  {\bf Dedication:} The authors wish to dedicate this paper to the
  memory of Professor Andrzej Skowro\'nski.
  As leader of the representation theory group in Toru\'n,
  Professor Skowro\'nski has consistently supported both authors,
  in particular 
  through conference invitations and travel support.  They are grateful
  for Professor Skowro\'nski's interest in their work and his
  helpful feedback.
  Both are saddened by the untimely loss of a treasured colleague
  and a good friend.


\bigskip
Address of the authors:

\parbox[t]{5.5cm}{\footnotesize\begin{center}
              Faculty of Mathematics\\
              and Computer Science\\
              Nicolaus Copernicus University\\
              ul.\ Chopina 12/18\\
              87-100 Toru\'n, Poland\end{center}}
\parbox[t]{5.5cm}{\footnotesize\begin{center}
              Department of\\
              Mathematical Sciences\\ 
              Florida Atlantic University\\
              777 Glades Road\\
              Boca Raton, Florida 33431\end{center}}

\smallskip \parbox[t]{5.5cm}{\centerline{\footnotesize\tt justus@mat.umk.pl}}
           \parbox[t]{5.5cm}{\centerline{\footnotesize\tt markus@math.fau.edu}}

\end{document}